\documentclass[12pt]{amsart}
\usepackage{amsmath}
\usepackage{amsthm}
\usepackage{amsfonts} 
\usepackage{hyperref}
\usepackage{epsfig}

\numberwithin{equation}{section}

\allowdisplaybreaks

\newtheorem{theorem}{Theorem}[section]
\newtheorem{proposition}{Proposition}[section]
\newtheorem{lemma}{Lemma}[section]
\newtheorem{remark}{Remark}[section]
\newtheorem{definition}{Definition}[section]
\newtheorem{corollary}{Corollary}[section]
\newtheorem{conjecture}{Conjecture}[section]

\renewcommand{\epsilon}{\varepsilon}

\newcommand{\abs}[1]{\left\vert #1\right\vert}
\newcommand{\1}[1]{{\mathbf 1}{\{#1\}}}
\newcommand{\R}{\mathbb{R}}

\newcommand{\Z}{\mathbb{Z}}
 
\newcommand{\PR}{\mathbb{P}}

\newcommand{\ES}{\mathbb{E}}

\title[]{Local trapping for elliptic random walks in random environments in $\Z^d$}
\date{}
\author[A.~Fribergh]{Alexander Fribergh}
\address{Alexander Fribergh\\Universit\'e de Montr\'eal, DMS\\
Pavillon Andr\'e-Aisenstadt\\     2920, chemin de la Tour Montréal (Qu\'ebec),  H3T 1J4} \email{fribergh@dms.umontreal.ca}
\author[D.~Kious]{Daniel Kious}
\address{Daniel Kious\\Ecole Polytechnique F\'ed\'erale de Lausanne\\ EPFL SB MATHAA PRST\\
MA B1 537, Station 8\\
CH-1015 Lausanne, Switzerland} \email{daniel.kious@epfl.ch}

\keywords{Random walk in random environments,  Ballisticity, Ellipticity} \subjclass[2000]{primary 60K37;
secondary 82D30}

\begin{document}

\begin{abstract}
 We consider elliptic random walks in i.i.d.~random environments on $\Z^d$. The main goal of this paper is to study under which ellipticity conditions local trapping occurs. Our main result is to exhibit an ellipticity criterion for ballistic behavior which extends previously known results. We also show that if the annealed expected exit time of a unit hypercube is infinite then the walk has zero asymptotic velocity.  
  \end{abstract}

\maketitle

\section{Introduction}

In this paper, we consider random walks in i.i.d.~random environments on $\Z^d$ for $d\geq 2$, in the specific case where the walk is directionally transient. It is expected that if the transition probabilities are uniformly elliptic then the walk is ballistic (see~\cite{Topics} and~\cite{Zeitouni}). This conjecture has been proved under stronger transience assumptions known as Sznitman's conditions $(T)$, $(T')$ or $(T)_{\gamma}$ (see~\cite{Szn01} and~\cite{Szneff}) and more recently condition $(P)_M$ (see~\cite{BDR}). All those transience conditions are believed to be equivalent under uniform ellipticity (see~\cite{Topics} and~\cite{Zeitouni}). Proving this equivalence is one of the major open problems in random walk in random environments (RWRE). We will give more details on these results in Section~\ref{sect_questions}.

If we remove the uniform ellipticity assumption, the walk may become sub-ballistic even in the elliptic setting (see~\cite{SabotTournier}, \cite{Sabot2}, \cite{Sabot1} and~\cite{Bouchet13}). This naturally raises the following question: which ellipticity conditions characterize a ballistic behavior? 

Recently, new ellipticity criteria for ballistic behavior have been proved (in~\cite{CR} and~\cite{BRS}). In this paper, we find a criterion (see Theorem~\ref{main}) for positive speed which extends previously known results. We believe that this criterion is close to optimal and we use it to exhibit new examples of ballistic random walks (see Proposition~\ref{new_example}). We also prove, under stronger assumptions, annealed and quenched central limit theorems (see Theorem \ref{cor_TCL}). Furthermore, we show that if the annealed expected exit time of a unit hypercube is infinite then the walk has zero asymptotic velocity (see Theorem~\ref{lb_tail}). We think that this criterion actually characterizes the zero-speed regime.

\subsection{Definition of the model}

Let us now define the model more precisely. Call $U$ the set of $2d$ canonical unit vectors and let
$$\mathcal{P}=\{(p(e))_{e\in U}:\ p(e)\geq 0\text{ for all } e\in U\text{, and }\ \sum_{e\in U} p(e)=1\}.$$
Fix some unit vector $\ell \in S^{d-1}$ and let us enumerate $U$ in the following manner: denote $\nu=\{e_1,\ldots,e_{d}\}$ an orthonormal basis of $\Z^d$ such that $e_1\cdot \ell \geq e_2 \cdot \ell \geq \cdots \geq e_d\cdot \ell \geq 0$ and set $e_{i+d}=-e_i$ for $i\in [1,d]$. In particular, Pythagoras's theorem implies that
\begin{equation}\label{e1ell}
e_1\cdot \ell \geq \frac 1{\sqrt d}.
\end{equation}

An environment $\omega$ is an element of $\Omega=\mathcal{P}^{\Z^d}$, which we view as a collection of transition probabilities $p^{\omega}(x,\cdot)=(p^{\omega}(x,e))_{e\in U}$ assigned to every vertex $x\in \Z^d$.

The random walk in the environment $\omega$ started from $x$ is the Markov chain $(X_n)_{n \geq 0}$ in $\Z^d$ with the law $P_x^{\omega}$ defined by $P_x^{\omega}[X_0=x]=1$ and 
\[
P_x^{\omega}[X_{n+1}=x+e \mid X_n=x]=p^{\omega}(x,e),
\]
for any $x\in \Z^d$ and $e\in U$. The law $P_x^{\omega}$ is commonly referred to as the \emph{quenched law}.

Let us consider ${\bf P}$ a probability measure on the environment space $\Omega$ which is a product measure, meaning that all random variables $p^{\omega}(x,\cdot)$ for $x\in \Z^d$ are i.i.d.~under ${\bf P}$. This allows us to define the averaged or \emph{annealed law} of the RWRE started at $x$ by defining $\PR_x=\int P_x^{\omega}d{\bf P}$. In the case where $x=0$, we will abbreviate $P_x^{\omega}$ and $\PR_x$ by $P^{\omega}$ and $\PR$ respectively.

We say that the environment  is elliptic if it verifies the following hypothesis $(E)$
\begin{equation}\label{hyp_E}
\text{for any $x\in \Z^d$ and $e\in U$ we have }   p^{\omega}(x,e)>0,\ \omega\text{-}{\bf P}\text{-a.s.}
\end{equation}
and we call an environment uniformly elliptic if there exists $\kappa>0$ such that
\begin{equation}\label{hyp_UE}
\text{for any $x\in \Z^d$ and $e\in U$ we have }   p^{\omega}(x,e)>\kappa,\ \omega\text{-}{\bf P}\text{-a.s.},
\end{equation}
a condition commonly denoted $(UE)$.

Given $\ell \in S^{d-1}$, we say that a RWRE is transient in the direction $\ell$ if $\PR[A_{\ell}]=1$ where $A_{\ell}=\{\displaystyle{\lim_{n \to \infty}} X_n \cdot \ell=\infty \}$.

We say that a RWRE is ballistic in the direction $\ell$ if
\[
\liminf_{n\to \infty} \frac{X_n \cdot \ell} n >0, \qquad \PR-\text{a.s.}
\]

\subsection{Former results and open questions}\label{sect_questions}

The case of RWRE on $\Z$ is well understood. In~\cite{Sol75}, the author identifies conditions that characterize recurrence versus (directional) transience, as well as zero-speed versus positive speed regimes. In particular, a regime of directional transience with zero-speed is exhibited. The existence of this regime is due to the existence of traps slowing down the walk down (see~\cite{ESZ} for details on trapping in RWREs on $\Z$). These traps can be formed even when transitions probabilities are uniformly elliptic.

In $\Z^d$, for $d\geq2$, it is more difficult to create traps. Actually, one of the main open problems concerning random walks in random environments is the following conjecture (see~\cite{Topics} and~\cite{Zeitouni}).
\begin{conjecture}\label{conjecture}
For any $\ell \in S^{d-1}$, we consider a random walk in a uniformly elliptic i.i.d.~environment in $\Z^d$ for $d\geq 2$. If it is transient in the direction $\ell$, then it is ballistic in the direction $\ell$.
\end{conjecture}

Let us discuss this conjecture on a very basic level. We can notice that there are two  main hypotheses in this conjecture.
\begin{enumerate}
\item The directional transience, which is a  ``global'' hypothesis on the transition probabilities. It gives information on how the walk explores the space.
\item The uniform ellipticity, which is a \lq\lq local\rq\rq~property for the transition probabilities. It provides us with a sufficient condition to avoid that the walk gets trapped in a small part of the environment.
\end{enumerate}

The main difficulty in proving Conjecture~\ref{conjecture} is to understand how the walk explores the space. Roughly speaking, we need to show that directional transience implies that the walk goes relatively directly in the direction $\ell$, i.e.~without zig-zagging on large scales. This, coupled with the fact that the walk cannot be trapped locally (because of uniform ellipticity) should imply that the walk is ballistic.

Surprisingly, it turns out to be technically difficult to show that a directional transient walk goes fairly directly in the direction $\ell$.   Conjecture~\ref{conjecture} has only been proved under stronger transience assumptions, under which we are given quantitative estimates on the exit probabilities of large slabs. Let us now introduce one of these conditions known as Sznitman's $(T)_{\gamma}^{\ell}$ (see~\cite{Szn01}).

For any set of vertices $A\subset \Z^d$,  we introduce the exit time of the set $A$ as
\[
 T_A^{\text{ex}}=\inf\{n\geq 0;~X_n\notin A\}.
\]

For any $\ell \in S^{d-1}$ and for any $b>0$, we define the slab
\[
U_b^{\ell}(L)=\{x\in \Z^d,\ -bL\leq x\cdot \ell \leq L\}.
\]

Set $\ell\in S^{d-1}$, $\gamma\in (0,1)$ and $b>0$, we say that the walk verifies the condition $(T)_{\gamma}^{\ell}$ if there exists a neighborhood $V\subset S^{d-1}$ of $\ell$ such that for all $\ell'\in V$, we have 
\begin{equation}\label{hyp_T}
\limsup_{L\to \infty} \frac 1 {L^{\gamma}} \ln \PR\bigl[X_{T_{U_b^{\ell'}(L)}^{\text{ex}}} \cdot \ell' <0\bigr] <0.
\end{equation}

Loosely speaking, this means that the probability of exiting a large slab against the asymptotic direction of the walk decays like a stretched exponential of exponent $\gamma$ (in the size of the slab).

Condition $(T)^{\ell}$ corresponds to condition $(T)_{\gamma}^{\ell}$ in the case where $\gamma=1$. Condition $(T')^{\ell}$ is defined as the fulfillment of condition $(T)_{\gamma}^{\ell}$ for all $\gamma\in (0,1)$. It was proved in~\cite{Szneff} that a random walk in i.i.d.~uniformly elliptic environment satisfying $(T')^{\ell}$ is ballistic in the direction $\ell$. It was also shown (see~\cite{Szneff}) that if $\gamma\in (1/2,1)$ then $(T)_{\gamma}^{\ell}$ implies $(T')$. 

Subsequent works (\cite{DR1},\cite{DR2} and~\cite{BDR}) have weakened the transience conditions that we can verify to prove ballistic behavior under uniform ellipticity. At this point in time, the state of the art is a result from~\cite{BDR} called polynomial condition typically denoted $(P)_M$. 

 To define this condition, let us consider for each $L$, $L'$, $\widetilde{L}>0$ and $\ell \in S^{d-1}$ the box
\[
B^{\ell}_{L,L',\widetilde{L}}=R\Bigl((-L',L)\times(-\widetilde{L},\widetilde{L})^{d-1}\Bigr) \cap \Z^d,
\]
where $R$ is the rotation of $\R^d$ with center $0$ which sends $e_1$ onto $\ell$. For $M\geq 1$ and $\ell\in S^{d-1}$, we say that the walk verifies condition $(P)_M^{\ell}$ if for all $L\geq \frac 23 3^{29d}$, there exist $L'\leq \frac 54 L$ and $\widetilde{L}\leq 72 L^3$ such that 
\begin{equation}\label{def_pm}
\PR[X_{T^{\text{ex}}_{B^{\ell}_{L,L',\widetilde{L}}}} \cdot \ell <L] \leq \frac 1{L^M}.
\end{equation}

This condition can be verified in a finite box, that is why it is referred to as an effective criterion. It should be noted that this condition obviously follows from tail estimates on the exit probabilities appearing in~(\ref{hyp_T}).

The main result of~\cite{BDR} is that for a RWRE in i.i.d.~environment with uniformly elliptic transition probabilities then $(P)_M$ for $M\geq 15d+5$ implies $(T')$. In particular, this implies ballisticity.

As we can see there has been a great deal of effort to understand under which transience assumptions we are able to prove ballistic behavior. But it is only recently that there have been developments on RWREs that are not uniformly elliptic. 

It is known (\cite{SabotTournier}, \cite{Sabot2}, \cite{Sabot1} and~\cite{Bouchet13}) that, in dimension $d\geq 2$, there exist elliptic random walks which are directionally transient but are not ballistic. More recently it has been shown, in~\cite{CR}, that under certain ellipticity conditions the polynomial condition $(P)_M$ is equivalent to $(T)'$. To be more specific, consider a RWRE in an elliptic i.i.d.~environment, we say that it verifies condition $(E)_0$ if 
\begin{equation}\label{def_e}
\text{for all $e\in U$ there exists $\eta_e>0$ such that }{\bf E}[(p^{\omega}(0,e))^{-\eta_e}]<\infty.
\end{equation}

One of the main results (Theorem 1.1) of~\cite{CR} is that if a random walk in an elliptic i.i.d.~environment  verifies $(P)_M^{\ell}$ for some $M\geq 15d+5$ and $(E)_0$ then this RWRE verifies $(T')^{\ell}$. We give an exact statement of this result in Theorem \ref{PequivT}. Furthermore, the authors of~\cite{CR} introduce sufficient ellipticity conditions for ballistic behavior under condition $(P)_M$. Later on, the ellipticity conditions for ballistic behavior were improved in~\cite{BRS}, providing an optimal criterion for the case of Dirichlet environments. See Section~\ref{sect_crit_pow} for details on this ellipticity condition.

In order to understand which ellipticity criteria characterize ballistic behavior we need to understand exactly how local traps are created. This is the main focus of this paper. After investigating how traps are created, it is our belief that a walk is ballistic if, and only if, the expected annealed exit time of a unit hypercube is finite. In order to back up our belief we prove the following:
\begin{enumerate}
\item if the annealed exit time of a unit hypercube is infinite then the walk has zero asymptotic velocity (see Theorem~\ref{lb_tail}),
\item we give a criterion for positive speed (see Theorem~\ref{main}). In order to verify this criterion it is sufficient to prove that we can exit some particular unit hypercube containing the origin. As we explain in Section~\ref{sect_why}, we believe that this criterion essentially means that the exit time of a unit hypercube has finite annealed expectation.
\end{enumerate}

For the aforementioned reasons, we believe that our positive speed criterion is near optimal. 

One of the contribution of this work is to bring forth the idea that the smallest possible traps are contained in unit hypercubes. This is striking since, in the reversible context, it is known (see~\cite{Fri11}) that if a walk is sub-ballistic then it can get trapped on just one edge. In Proposition~\ref{ex_nul}, we provide an example of a sub-ballistic RWRE than cannot  stay long on only one edge.

\subsection{Plan of the article}

Let us present how this paper is structured.

In the next section (Section~\ref{sect_not}), we will start by introducing some basic notations as well as facts about regeneration times. This is a central tool for determining whether or not a walk is ballistic.

 After that, in Section~\ref{sect_results}, we will present our zero speed criterion (see Theorem~\ref{lb_tail}) and our positive speed criterion (see Theorem~\ref{main}). We also state annealed and quenched central limit theorems (see Theorem \ref{cor_TCL}). 
 
 Before moving on to proofs, we discuss the intuition behind our main results in Section~\ref{sect_discussion}. In this section we try to justify why our criterion is close to optimal. We also provide a new example of ballistic walks (see Proposition~\ref{new_example}) and a zero-speed random walk that can never stay long on only one edge (see  Proposition~\ref{ex_nul}).

The proof for the sufficient condition for positive speed is presented in Section~\ref{sect_thm_proof}. This section is divided into three parts. The first one is Section~\ref{sect_att} in which we prove the key estimate Proposition~\ref{reachlog}: under our ballisticity criterion  the quenched probability of reaching a point far away is lower bounded. The second section is Section~\ref{sect_ref_res}, in which we recall some classical results from RWREs. Finally the third part is Section~\ref{sect_tailtau}, in which we finish the proof by providing an upper bound on the tail of the first regeneration time.\\

Finally we present the proof of the sufficient condition for zero speed in Section~\ref{sect_zero_speed}. This section is essentially independent of the rest of the paper.

Before moving on to the rest of the paper, let us specify that in the course of our proofs, $c$ and $C$ will typically denote constants in $(0,\infty)$ whose value may change from line to line.

\section{Basic notations and regeneration times}\label{sect_not}

In this section, we introduce some basic notations and we summarize the facts we need about regeneration times.

Let us define the adjacency $\sim$ such that, for $x,y\in\mathbb{Z}^d$, we have $x\sim y$ if and only if $\|x-y\|_1=1$. Given a set $V$ of vertices of $\Z^d$, we denote by $\abs{V}$ its cardinality, by $E(V)=\{ [x,y] :  x\sim y,\ x,y\in V\}$ its edges and 
\[
\partial V= \{x \notin V:~\exists y\in V,~x\sim y\},
\]
its border. For $A\subset \Z^d$ and $x\in A$, we denote
\[
\partial_x A=\{y\in \partial A:~x\sim y\},
\]
the neighbors of $x$ which are outside of $A$.

For any  $r>0$, we denote
\begin{equation}\label{def_HH}
\mathcal{H}^+(r)=\{z\in \Z^d, z\cdot \ell > r\}.
\end{equation}

For any set of vertices $A\subset \Z^d$,  we introduce the hitting times
\[
T_A=\inf\{n\geq 0;~X_n\in A\},~T_A^+=\inf\{n\geq 1;~X_n\in A \}.
\]

We will use a slight abuse of notation and write $x$ instead of $\{x\}$ when the set is a point $x$.

\subsection{Regeneration times}\label{sect_regen}

We set $a\in (2\sqrt d, 10 \sqrt d)$ and define
\[
D=\inf\{n\geq 0: X_n\cdot \vec{\ell}<X_0\cdot \vec{\ell}\},
\]
as well as the stopping times $S_k$, $k\geq 0$, $R_k$, $k\geq 1$, and the levels $M_k$, $k\geq 0$:
\begin{align*}
S_0&=0,~M_0=X_0\cdot \vec{\ell},\text{ and for $k\geq 0$}, \\
S_{k+1}&=T_{\mathcal{H}^+(M_k+a)}\leq \infty,~R_{k+1}=D\circ\theta_{S_{k+1}}+S_{k+1}\leq \infty, \\
M_{k+1}&=\sup_{n\leq R_{k+1}} X_n\cdot \vec{\ell}\leq \infty,
\end{align*}
where $\theta_{\cdot}$ is the shift operator.

Finally, we define the basic regeneration time
\[
\tau_1=S_K, \text{ with } K=\inf\{k \geq 1:~ S_k<\infty \text{ and } R_k=\infty\}.
\]

\begin{remark} The choice of $a\in(2\sqrt d,10 \sqrt d)$ is only necessary to prove the non degeneracy of the covariance matrix in Theorem~\ref{speed_regen}.\end{remark}

It follows from directional transience (see for example~\cite{SznZer}) that 
\begin{equation}\label{pos_D}
\PR[D=\infty]>0,
\end{equation}
this allows us to define
\begin{equation}
\label{def_prob_regen}
\PR[~\cdot \mid 0-\text{regen}]=\PR[~\cdot \mid D=\infty].
\end{equation}

Then let us define the sequence $\tau_0=0<\tau_1<\tau_2< \cdots <\tau_k< \cdots$ (these inequalities hold except if the regeneration times are infinite), via the following procedure:
\begin{equation}
\label{regen_struct}
\tau_{k+1}=\tau_1+\tau_k(X_{\tau_1+ \cdot }-X_{\tau_1},~\omega(\cdot+X_{\tau_1})),~k\geq 0.
\end{equation}

That is, the $(k+1)$-th regeneration time is the $k$-th such time after the first one.

The first main result is that the regeneration structure exists and is finite (see for example~\cite{SznZer}).
\begin{lemma}
\label{tau_finite}
Let us consider a RWRE in an elliptic i.i.d.~environment. Fix $\ell\in S^{d-1}$ and assume that the random walk is transient in the direction $\ell$. For any $k\geq 1$, we have ${\bf P}$-a.s., for all $x\in \Z^d$,
\[
\tau_k<\infty, \qquad \qquad P_x^{\omega}\text{-a.s.}
\]
\end{lemma}

The fundamental renewal property is now stated (see for example~\cite{SznZer})
\begin{theorem}
\label{tau_indep}
Let us consider a RWRE in an elliptic i.i.d.~environment. Fix $\ell\in S^{d-1}$ and assume that  the random walk is transient in the direction $\ell$.

Under $\PR$, the processes $(X_{\tau_1\wedge\cdot}), (X_{(\tau_1+\cdot)\wedge \tau_2}-X_{\tau_1}),\cdots, (X_{(\tau_k+\cdot)\wedge \tau_{k+1}}-X_{\tau_k}),\ldots$ are independent and, except for the first one, are distributed as $(X_{\tau_1\wedge \cdot} )$ under $\PR[~\cdot \mid 0-\text{regen}]$.
\end{theorem}


The previous results we mention imply the following Theorem (see~\cite{SznZer}, \cite{Szn00} and~\cite{Z}).
\begin{theorem}\label{speed_regen}
Let us consider a RWRE in an elliptic i.i.d.~environment. Fix $\ell\in S^{d-1}$ and assume that the random walk is transient in the direction $\ell$. Then there exists a limiting deterministic velocity 
\[
 \lim_{n\to \infty}\frac{X_n}n = v, \qquad \PR \text{-a.s.,} 
 \]
 where
 \[
v=\frac{\ES[X_{\tau_1}\mid 0-\text{regen}]}{\ES[\tau_1\mid 0-\text{regen}]},
\]
even in the case where $\ES[\tau_1\mid 0-\text{regen}]=\infty$. In particular one can obtain that
\[
\text{ if }\ES[\tau_1\mid 0-\text{regen}]<\infty\text{ then }v>0.
\]

Furthermore, if $\ES[\tau_1^2\mid 0-\text{regen}]<\infty$, then
\[
\epsilon^{1/2}\bigl(X_{\lfloor \epsilon^{-1}n\rfloor}-\lfloor \epsilon^{-1} n\rfloor v\bigr),
\]
converges in law under $\PR$ to a brownian motion with a non-degenerate covariance matrix.
 \end{theorem}

\section{Results}\label{sect_results}

\subsection{A criterion for zero-speed}

We call unit hypercube located at $x$ the set
\begin{align}\label{def_hx}
\mathfrak{H}_x=\Bigl\{x+\sum_{i=1}^d \epsilon_i e_i,\ \text{where } \epsilon_i\in \{0,1\} \text{ for all } i\in [1,d]\Bigr\}.
\end{align}

For simplicity we use $\mathfrak{H}_0=\mathfrak{H}$. Let us denote $(H)_{\alpha}$ the following hypothesis
\begin{equation}\label{hyp_H}
\max_{x\in \mathfrak{H}} \ES_x\Bigl[\Bigl(T^{\text{ex}}_{\mathfrak{H}}\Bigr)^\alpha\Bigr]=\infty.
\end{equation}

In the next theorem we exhibit a criterion for zero-speed. We believe that criterion to be sharp.
\begin{theorem}\label{lb_tail}
Let us consider a RWRE in an elliptic i.i.d.~environment. Fix $\ell\in S^{d-1}$ and assume that the random walk is transient in the direction $\ell$. 

If $(H)_1$ is verified, then the walk has zero speed, i.e.~$v=\vec{0}$.
\end{theorem}

In the same way that we prove Theorem~\ref{lb_tail} (see~(\ref{proof_rem})), we can obtain lower bound estimates on regeneration times (see Section~\ref{sect_regen} for a precise definition of regeneration times).

\begin{remark}Let us consider a RWRE in an elliptic i.i.d.~environment. Fix $\ell\in S^{d-1}$ and assume that the random walk is transient in the direction $\ell$. Furthermore, we assume that there exists $\alpha>0$ such that we have $(H)_{\alpha}$. Then
\[
\ES[\tau_1^{\alpha}\mid 0-\text{regen}] =\infty.
\]
We believe that this last display is equivalent to $(H)_1$ when $\alpha=1$.
\end{remark}

\subsection{A positive speed criterion}

Let $\mathcal{C}$ be a unit hypercube of $\Z^d$ and $y\in \mathcal{C}$. We denote
\begin{equation}\label{def_Q}
Q_{y}^{\mathcal{C}}=\max_{z\in \partial_y \mathcal{C}} p^{\omega}(y,z-y),
\end{equation}
the highest probability leading out of $\mathcal{C}$ from $y$.

Let $H$ be a unit hypercube of $\Z^d$ and $x\in \mathcal{C}$. We denote for any $y\in \mathcal{C}$
\begin{equation}\label{def_Qtilde}
\widetilde{Q}_{x,y}^{\mathcal{C}}=P^{\omega}_x[T_{\partial \mathcal{C}} <T_x^+,X_{T_{\partial \mathcal{C}}}\in\partial_y \mathcal{C}]
\end{equation}
the  probability starting from $x$ to exit $\mathcal{C}$ via a neighbor of $y$ before returning to $x$.

In order to state our main result (Theorem~\ref{main}), which is our criterion for positive speed, we need to introduce the concept of Markovian hypercube which we define in the next section.

\subsubsection{Markovian hypercube}\label{sect_markhyp}

We denote $\overline{\mathfrak{H}}_0=\{\mathfrak{H}_x, \text{ where }x\in \Z^d\text{ and } 0\in \mathfrak{H}_x\}$, the sets of unit hypercubes containing $0$.

Let us  introduce the notion of hypercube discovered in a Markovian fashion. It is a function $h$ from $\Omega$ into $\overline{\mathfrak{H}}_0$ constructed in a particular manner that we are going to describe below.

We will start by introducing some notations before explaining intuitively the construction of a Markovian hypercube. We construct recursively functions $f_0,\ldots, f_{2^d-1}$ from $\Omega$ into $\Z^d$, such that
\begin{enumerate}
\item $f_0(\omega)=0$ ${\bf P}$-a.s.,
\item for $i\geq 0$, the function $f_{i+1}$ is measurable with respect to $\{p^{\omega}(x,\cdot), x\in \{f_0(\omega),\ldots, f_i(\omega)\}\}$,
\item for any $i\geq 0$, $f_{i+1}(\omega) \in \partial \{f_0(\omega),\ldots, f_i(\omega)\}$ ,
\item for any $i\geq 0$, ${\bf P}$-a.s.~there exists $H(\omega) \in \overline{\mathfrak{H}}_0$ such that we have $\{f_0(\omega),\ldots, f_i(\omega)\} \subset H(\omega)$.
\end{enumerate}

In words, this means we start from $0$ then, using the information given by the transition probabilities at $0$, we choose to add a site called $f_1(\omega)$.  Then we use the information given by the transition probabilities at $0$ and $f_1(\omega)$ to add a new adjacent site called $f_2(\omega)$. We continue this procedure recursively, with the only restriction that we can never add a point $f_{i+1}(\omega)$ such that the points $\{f_0(\omega),\ldots,f_{i+1}(\omega)\}$ would not be included in a unit hypercube. 

This procedure yields a hypercube $\{f_0(\omega),\ldots,f_{2^d-1}(\omega)\}$ containing $0$. A hypercube constructed in this way is said to be discovered in a Markovian fashion. Given such a hypercube $h(\omega)$, we denote $x_0(\omega)$ the only point in $\Z^d$ such that 
\begin{align}\label{def_x0}
\{x_0(\omega)+y \text{ with } y\in \mathfrak{H}\}=h(\omega).
\end{align}

A couple of functions  $(h(\omega),(\alpha_x(\omega))_{x\in \mathfrak{H}})$ is  called a marked Markovian hypercube if
\begin{enumerate}
\item $h(\omega)$ is a hypercube discovered in a Markovian fashion,
\item for any $x\in \mathfrak{H}$, the function $\alpha_x(\omega)$ goes from $\Omega$ into $\R^+$,
\item for any $x\in \mathfrak{H}$, the function $\alpha_x(\omega)$ is measurable with respect to $\{p^{\omega}(y,\cdot), y\in h(\omega)\}$.
\end{enumerate}

In a marked Markovian hypercube, we simply add, using the information given by the transition probabilities in $h(\omega)$, certain marks in $\R^+$ to every corner of the hypercube $h(\omega)$. We can see this by associating the mark $\alpha_x(\omega)$ to the corner $x_0(\omega)+x\in h(\omega)$ for every $x\in \mathfrak{H}$.

\begin{remark}\label{rem_markov}
It can easily be seen from the definition that a marked Markovian hypercube $(h(\omega),(\alpha_x(\omega))_{x\in \mathfrak{H}})$ is measurable with respect to $\{p^{\omega}(y,\cdot), y\in h(\omega)\}$. This means that a marked Markovian hypercube can be determined independently of the information outside of that hypercube.
\end{remark}

\subsubsection{Criterion $(K)_\alpha$}

Recalling the definitions at~(\ref{def_Q}) and~(\ref{def_Qtilde}), 

\begin{definition}\label{def_k}
We denote $(K)_{\alpha}$ the following hypothesis: 
\begin{enumerate}
\item there exists $\gamma_x\in \R^+$, for every $x\in \mathfrak{H}$, such that we have
\[
{\bf E}\Bigl[\Bigl(Q_x^{\mathfrak{H}}\Bigr)^{-\gamma_x}\Bigr]<\infty \qquad \text{for all $x\in \mathfrak{H}$},
\]
\item there exists a marked Markovian hypercube $(h(\omega),(\alpha_x(\omega))_{x\in \mathfrak{H}})$  such that 
\[
{\bf E}\Bigl[\prod_{x\in \mathfrak{H}} \Bigl(\widetilde{Q}_{0,x_0(\omega)+x}^{h(\omega)}\Bigr)^{-\alpha_x(\omega)}\Bigr]<\infty,
\]
\item there exists $\epsilon>0$ such that 
\[
\sum_{x\in \mathfrak{H}} (\gamma_x \wedge \alpha_x(\omega)) \geq \alpha+\epsilon \qquad {\bf P}\text{-a.s}.
\]
\end{enumerate}
\end{definition}

This condition may seem very complicated. This is why, in Section~\ref{sect_discussion}, we shall dedicate a few pages to explaining the meaning of this condition and how to apply it. In particular, we will justify why the conditions involving the exponents $\gamma_i$ are  verified in the positive speed regime under some regularity properties of the tails at $0$ of $Q_x^{\mathfrak{H}}$ for $x\in \mathfrak{H}$ (see Lemma~\ref{cond_gamma_easy} and below).

\subsubsection{Criterion for positive speed}

The next result proves that, under sufficiently strong transience conditions, the condition $(K)_1$ (see Definition~\ref{def_k}) and $(E)_0$ (defined at~(\ref{def_e})) imply positive speed.
\begin{theorem}\label{main}
Let us consider a RWRE in an elliptic i.i.d.~environment that verifies conditions $(E)_0$ and $(P)_M^{\ell}$ for some $M\geq 15d+5$ and $\ell \in S^{d-1}$. If furthermore condition $(K)_1$ is verified, then  the walk is ballistic in the direction $\ell$, i.e.
\[
\lim_{n\to \infty} \frac{X_n}n=v \text{ where } v\cdot \ell>0.
\]
\end{theorem}

In the course of the proof of Theorem~\ref{main}, we obtain tail estimates on $\tau_1$, see Proposition~\ref{tailtau}.

Although the criterion $(K)_1$ is very flexible, it can be a bit cumbersome to verify it. However, in concrete examples (for example, see Proposition \ref{new_example}, which exhibits new examples of ballistic walks), we can use a much simpler criterion $(\widetilde{K})_1$ defined by:
\[
\min_{x\in \mathfrak{H}} {\bf E}\Bigl[\Bigl(Q_x^{\mathfrak{H}}\Bigr)^{-(1+\epsilon)}\Bigr]<\infty\text{, for some } \varepsilon>0.
\]
\begin{remark}\label{suff_cond} Condition $(\widetilde{K})_1$ is easily seen to imply $(K)_1$, by choosing the $\gamma$'s and the Markovian hypercube conveniently. Indeed, assume $(\widetilde{K})_1$ is verified and denote $x_{\min}\in\mathfrak{H}$ the vertex for which the minimum is reached. Then, define, for any $x\in\mathfrak{H}$, $\alpha_x(\omega)=\gamma_x=(1+\varepsilon)\1{x=x_{\min}}$ and, recalling \eqref{def_x0}, let $h(\omega)$ be such that $x_0(\omega)=-x_{\min}$.\\
This means that if a RWRE is in an elliptic i.i.d.~environment verifies conditions $(\widetilde{K})_1$, $(E)_0$ and $(P)_M^{\ell}$ for some $M\geq 15d+5$ and $\ell \in S^{d-1}$, then  the walk is ballistic in the direction $\ell$.
\end{remark}


\subsubsection{Central limit theorems} Under some stronger assumptions, we prove an annealed central limit theorem, using Theorem \ref{speed_regen}, and also a quenched central limit theorem, using the main result of Bouchet, Sabot and dos Santos \cite{BDSS} (improving on previous results by Rassoul-Agha, Sepp\"al\"ainen \cite{RAS} and Berger, Zeitouni \cite{BZ}).

\begin{theorem}\label{cor_TCL}
Consider a RWRE in $\Z^d$ with $d\ge2$. Let $\ell\in S^{d-1}$ and $M\ge15d+5$. Assume that the random walk satisfies conditions $(P)_M^\ell$, $(E)_0$ and $(K)_2$, then we have an annealed central limit theorem, i.e.
\[
\epsilon^{1/2}\left(X_{\lfloor\epsilon^{-1}n\rfloor}-\lfloor\epsilon^{-1}n\rfloor v\right)
\]
converges in law under $\PR_0$ as $\epsilon\rightarrow0$ to a Brownian Motion with non-degenerate covariance matrix. Under the same conditions, we also have a quenched central limit theorem, i.e.~the previous expression also converges in law under $P_0^\omega$ as $\epsilon\rightarrow0$ to a Brownian Motion with non-degenerate covariance matrix for $\omega$-{\bf P}-a.s.
\end{theorem}

\section{Discussion on the main results}\label{sect_discussion}

Let us now address some questions the reader might have about hypothesis $(K)_{1}$.
\begin{enumerate}
\item What does this condition intuitively mean and why?
\item How do we apply our criterion? Why is this criterion general?
\item Why do unit hypercubes appear?
\end{enumerate}

Through the course of these explanations we hope to convince the reader that the hypothesis $(K)_1$ is a near optimal criterion for positive speed.

\subsection{What does this condition intuitively mean and why?}\label{sect_why}

We believe that $(K)_1$ essentially means that the expected annealed exit time of a hypercube has a moment of order $1+\epsilon$ for some $\epsilon>0$. This would mean that $(K)_1$ and $(H)_1$  cover most cases of RWREs and allow us to determine whether or not the walk has positive speed under the hypotheses $(E)_0$ and $(P)_M^{\ell}$.

Let us now explain why $(K)_1$ and $(H)_1$ are close to complementary. The complementary condition $(H)_1^c$ exactly means that
\begin{align}\label{def_Hc}
\ES_x\left[T_{\mathfrak{H}}^{\text{ex}}\right]<\infty\text{, for all } x\in\mathfrak{H}.
\end{align}

\subsubsection{Why part $(1)$ of $(K)_1$ is typically verified in the positive speed regime}

Recall that $\overline{\mathfrak{H}}_0$ is the set of all the hypercubes containing $0$. 

\begin{lemma}\label{cond_gamma_easy}
If $(H)_1^c$ holds, then for any  hypercube $H\in\overline{\mathfrak{H}}_0$, we have
\[
\mathbf{E}\left[\min_{y\in H}\left(Q^{H}_{y}\right)^{-1}\right]<\infty.
\]
\end{lemma}
The proof of this lemma is straightforward,
it follows from the fact that, on any point $y\in H$, the exit probability of $H$ is at most $\max_{y\in H} Q^{H}_{y}$.\\

For part $(1)$ of $(K)_1$, we require that  there exist $\epsilon>0$ and $\gamma_y\geq 0$ for all $y\in H$ such that
 \[
 \text{for all $y\in H$, we have }\mathbf{E}[\left(Q^H_y\right)^{-\gamma_y}]<\infty,
 \] 
 or, equivalently by the independence of  $Q^H_y$ for $y\in H$,
 \begin{align}\label{casoule2}
\mathbf{E}\Bigl[\prod_{y\in H}\left(Q^H_y\right)^{-\gamma_y}\Bigr]<\infty,
\end{align}
with $\displaystyle{\sum_{y\in H}} \gamma_y \geq1+\epsilon$.
 
In most generic cases, where the tails of all $Q^H_y$ for $y\in H$ are sufficiently smooth, e.g.~polynomial tails or Dirichlet environment, one can see that this condition is equivalent to $\mathbf{E}\Bigl[\displaystyle{\min_{y\in H}}\Bigl( Q^{H}_{y}\Bigr)^{-(1+\epsilon)}\Bigr]<\infty$. This is very similar to the condition in Lemma~\ref{cond_gamma_easy}, although slightly stronger because of the $\epsilon$.

Besides, note that it is easy to see that \eqref{casoule2} implies that the condition in Lemma~\ref{cond_gamma_easy} is verified, indeed:
\[
\prod_{x\in H} \Bigl({Q}_{x}^{H}\Bigr)^{-\gamma_x}\geq \prod_{x\in H} \Bigl(\min_{y\in H}\frac 1 { {Q}_{y}^{H}}\Bigr)^{\gamma_x}\geq \min_{y\in H}\frac 1 { {Q}_{y}^{H}},
\]
where we used that for any $x\in H$ we have $\min_{y\in H}\frac 1 { {Q}_{y}^{H}}\leq  \frac 1 { {Q}_{x}^{}}$.\\

Recalling Theorem~\ref{lb_tail}, we know that $(H)_1^c$ holds whenever the speed is positive, we hope to have convinced the reader that part $(1)$ of $(K)_1$ is typically verified in the positive speed regime.

\subsubsection{How does part $(2)$ of $(K)_1$ relate to the exit time of hypercubes}

Let us now explain why $(K)_1$ and $(H)_1$ are close to complementary.

The following proposition states a condition which is equivalent to $(H)_1^c$.
\begin{proposition}\label{casoule}
The condition $(H)_1^c$ holds if, and only if, for any  hypercube $H\in\overline{\mathfrak{H}}_0$, we have
\[
\mathbf{E}\left[\min_{y\in H}\left(\widetilde{Q}^{H}_{0,y}\right)^{-1}\right]<\infty.
\]
\end{proposition}
\begin{proof}
It will be sufficient to show that
\begin{align}\label{equiv_Hc}
\mathbf{E}\left[\min_{y\in\mathfrak{H}}\left(\widetilde{Q}^{\mathfrak{H}}_{x,y}\right)^{-1}\right]<\infty \text{, for all }  x\in\mathfrak{H},
\end{align}
which, by translation invariance of the environment, can be equivalently stated in the following way: for any hypercube $H$ containing $0$, we have
\[
\mathbf{E}\left[\min_{y\in H}\left(\widetilde{Q}^{H}_{0,y}\right)^{-1}\right]<\infty.
\]

For all $x\in \mathfrak{H}$, we define the number of visits to $x$ before exiting $\mathfrak{H}$ by
\[
N(x)=\sum_{n=0}^{T_{\mathfrak{H}}^{\text{ex}}}\1{X_n=x},
\]
and notice that
\begin{equation}\label{petitedefT}
T_{\mathfrak{H}}^{\text{ex}}=\sum_{x\in \mathfrak{H}} N(x).
\end{equation}

Define also, for any $x\in \mathfrak{H}$,
\[
\widetilde{Q}^{\mathfrak{H}}_x=\sum_{y\in \mathfrak{H}}\widetilde{Q}^{ \mathfrak{H}}_{x,y},
\]
which verifies that for any $y\in \mathfrak{H}$
\begin{equation}\label{ineq_Q}
\widetilde{Q}^{ \mathfrak{H}}_{x,y}\leq \widetilde{Q}^{\mathfrak{H}}_x \leq 2^d\max_{y\in \mathfrak{H}}\widetilde{Q}^{ \mathfrak{H}}_{x,y}.
\end{equation}

Now, for any $x\in\mathfrak{H}$ and any starting point $x_0\in\mathfrak{H}$ (could be the same), we get
\begin{align*}
E_{x_0}^\omega\left[N(x)\right]&=\sum_{n\ge1}P_{x_0}^\omega\left[N(x)\ge n\right]\\
&=P_{x_0}^\omega\left[T_x<T_{\mathfrak{H}}^{\text{ex}}\right]\sum_{n\ge0} \left(1-\widetilde{Q}^{\mathfrak{H}}_x\right)^n\\
&=\frac{P_{x_0}^\omega\left[T_x<T_{\mathfrak{H}}^{\text{ex}}\right]}{\widetilde{Q}^{\mathfrak{H}}_x}.
\end{align*}

In particular, we have, for any $x_0\in\mathfrak{H}$,
\[
\frac{1}{\widetilde{Q}^{\mathfrak{H}}_{x_0}}\le E_{x_0}^\omega\left[T_{\mathfrak{H}}^{\text{ex}}\right]\le \sum_{y\in\mathfrak{H}}\frac{1}{\widetilde{Q}^{\mathfrak{H}}_{y}},
\]
where the lower bound is obtained by keeping only the term for which $x=x_0$ in the sum in \eqref{petitedefT}.

Thus, $(H)_1^c$, defined in \eqref{def_Hc}, holds if, and only if,
\begin{align*}
\mathbf{E}\left[\frac{1}{\widetilde{Q}^{\mathfrak{H}}_{x}}\right]<\infty \text{, for all }  x\in\mathfrak{H},
\end{align*}
which is also equivalent by~(\ref{ineq_Q}) to \eqref{equiv_Hc}.
\end{proof}


For technical reasons it is difficult for us to use the condition appearing in Proposition \ref{casoule}. Indeed, we want to use large deviations which requires slightly stronger assumptions. The way we strenghten the condition in Proposition~\ref{casoule} is similar to what we did in \eqref{casoule2}.
In this new case, the random variables $\widetilde{Q}^{H}_{0,y}$ are correlated. For this reason, we introduce the following condition which is slightly more flexible: there exist random variables $(\alpha_x(\omega))_{x\in H}$ such that
\begin{equation}\label{cond_3}
{\bf E}\Bigl[\prod_{x\in H} \Bigl(\widetilde{Q}_{0,x}^{H}\Bigr)^{-\alpha_x(\omega)}\Bigr]<\infty,
\end{equation}
with $\sum_{x\in H}  \alpha_x(\omega) \geq 1$, ${\bf P}$-almost surely. As the reader may notice that this is similar to  parts (2) and $(3)$ of condition $(K)_1$. On the one hand, we lost a constant $\epsilon$. On the other hand, we only require this condition \eqref{cond_3} to be verified on a Markovian hypercube instead of all hypercubes containing $0$.
Only having to verify this property on a single Markovian hypercube gives us a lot of flexibility. The flip side of this flexibility is that we require a slightly stronger condition on the $\alpha$'s (see part $(3)$ of $(K)_1$). This will be discussed in the next section.\\
At first glance, allowing our exponents $\alpha_x(\omega)$ to be random in condition $(K)_1$ may seem a bit odd. But this randomness gives us some extra flexibility and makes our condition more general. In particular, it allows us to very easily check that our new condition is more general than previous ones (see Proposition \ref{propmoreg}).

\subsubsection{Some comments on part $(3)$ of $(K)_1$}\label{sect_casoule}

The reader can easily realize that, because of part $(3)$ of $(K)_1$ , it is useless to choose $\alpha_x(\omega)>\gamma_x$ for any $x\in\mathfrak{H}$.

Such a condition is obviously needed, since part $(2)$ of $(K)_1$ can always be verified with $\sum_{x\in h(\omega)} \alpha_x(\omega)\ge1+\epsilon$. Indeed, using only the transition probabilities at $0$, we can always construct a Markovian hypercube $h(\omega)$ from which the walker can exit in one step with probability at least $1/(2d)$ through an edge $e(\omega)$. By assigning $\alpha_{e(\omega)}(\omega)=2$, we can verify part $(2)$ of $(K)_1$, but part $(3)$ is not necessarily verified.\\

Intuitively, part $(3)$ of $(K)_1$ prevents us from using too strongly the conditioning provided by $h(\omega)$. In particular, the tail at $0$ of $\widetilde{Q}_{0,x_0(\omega)+x}^{h(\omega)}$ cannot be much lighter than the one of $Q_x^{\mathfrak{H}}$.

In the next section, we will explain how to choose the Markovian hypercube in order to verify $(K)_\alpha$.

\subsection{How do we apply the criterion? Why is this criterion general?}\label{sect_crit_pow}

To apply the criterion $(K)_{\alpha}$, we need to find an efficient way of choosing our Markovian hypercube $h(\omega)$. Generally speaking one should try to choose the Markovian hypercube $h(\omega)$, in such a way that we can easily move around the hypercube. This will increase the potential exit points and make it easier to verify $(K)_{\alpha}$. Surprisingly one should not choose the hypercube from which it is the easiest to exit (see Section~\ref{sect_casoule}).

The choice of $\alpha_x(\omega)$ is supposed to reflect how easy it is to exit the hypercube $h(\omega)$ by the corner $x+x_0(\omega)$. The choice of $\alpha_x(\omega)=0$ means that we essentially ignore the possibility of exiting the hypercube in that corner.

In order to illustrate how to apply the criterion $(K)_{\alpha}$, we are going to show that $(K)_{\alpha}$ is more general than the current best criterion for positive speed (see~\cite{BRS}). This will be done in two parts. Firstly, we take the criterion for positive speed exhibited in~\cite{BRS} and show that it implies $(K_1)$. Secondly, we will provide an example which verifies $(K)_1$ but no former criterion.

{\bf Extending previous results}

In~\cite{BRS}, the authors introduced the following condition called $(E')_1$: there exists $\{\phi(e), e\in U\}\in (0,\infty)^{2d}$ such that
\begin{enumerate}
\item $2\sum_{e\in U} \phi(e)-\sup_{e\in U} (\phi(e)+\phi(-e))>1$,
\item for every $e\in U$ we have that
\[
{\bf E}\Bigl[\exp\Big(\sum_{e'\neq e} \phi(e')\log \frac 1{p^{\omega}(0,e')}\Big)\Bigr]<\infty,
\]
\end{enumerate}

It is shown (see Theorem 2 in~\cite{BRS}) that under $(E')_1$ the walk is ballistic provided the conditions $(P)_M$ and $(E)_0$ are verified. 

Our goal here is to show that the ellipticity condition we present in this paper is more general than those of~\cite{CR} and~\cite{BRS}.
\begin{proposition}\label{propmoreg}
Any random environment verifying the condition $(E')_1$ also verifies $(K)_1$.
\end{proposition}
 \begin{proof}
Assume that  there exists $\{\phi(e), e\in U\}\in (0,\infty)^{2d}$ such that $(E')_1$ holds. Then there exists $\epsilon>0$ such that
\begin{enumerate}
\item we have
\begin{equation}\label{eps}
2\sum_{e\in U} \phi(e)-\sup_{e\in U} (\phi(e)+\phi(-e))>1+\epsilon,
\end{equation}
\item for every $e\in U$ we have that
\begin{equation}\label{zzz}
{\bf E}\Bigl[\exp\Bigl(\sum_{e'\neq e} \phi(e')\log \frac 1{p^{\omega}(0,e')}\Bigr)\Bigr]<\infty,
\end{equation}
\end{enumerate}

Let us check that we can verify the three conditions of $(K)_1$ (defined at Definition~\ref{def_k}). This will prove our proposition.

\vspace{0.5cm}

{\it First condition}

\vspace{0.5cm}

The first point $(1)$ of Definition \ref{def_k} of $(K)_1$ holds by choosing for any $x\in\mathfrak{H}$
\begin{align}\label{res1}
\gamma_x=\sum_{e\in U: x+e\in\partial_x \mathfrak{H}} \phi(e),
\end{align}
because of property~(\ref{zzz}).

\vspace{0.5cm}

{\it The definition of the Markovian hypercube and the third condition}

\vspace{0.5cm}

Now, let us construct a marked Markovian hypercube $(h(\omega),(\alpha_x(\omega))_{x\in \mathfrak{H}})$ (see Section \ref{sect_markhyp}) fulfilling condition $(K)_1$.

Recall the definition of $\{e_1,...,e_d\}$ in page \pageref{e1ell} and $\mathfrak{H}_0$  at \eqref{def_hx}. Fix $\delta\in(0,1/(2d))$ and define the event
\[
A_1=\left\{p^\omega(0,e_1)\ge\delta\right\},
\]
then recursively, for all $k\in\{2,...,2d\}$,
\[
A_k=\left\{p^\omega(0,e_k)\ge\delta\right\}\setminus\left(\bigcup_{i=1}^{k-1}A_i\right),
\]
so that $(A_k)_{1\leq k\leq 2d}$ forms a partition of $\Omega$.

Now, we define a Markovian hypercube $h(\omega)$ such that $h(\omega)=\mathfrak{H}_0$ on $A_k$ for all $k\in\{1,...,d\}$, and $h(\omega)=\mathfrak{H}_{(-1,...,-1)}$ on $A_k$ for all $k\in\{1+d,...,2d\}$. Recall the definition $\eqref{def_x0}$ of $x_0(\omega)$ and notice that either $x_0(\omega)=0$ or $x_0(\omega)=(-1,...,-1)$.\\

Let us work on the event $A_k$, for some $k\in\{1,...,2d\}$. We will now label some vertices of the hypercube $h(\omega)$. Firstly, let $v_0^{(k)}=0$ be the origin. This vertex $v_0^{(k)}$ has $d$ neighbors in $h(\omega)$: let us call them $v_1^{(k)},...,v_d^{(k)}$ such that $v_d^{(k)}=e_k$. Notice that, on the event $A_k$, $p^\omega(v_0^{(k)},v_d^{(k)})=p^\omega(0,e_k)\ge\delta$. 

The vertex $v_d^{(k)}$ has also $d$ neighbors in $h(\omega)$, one of them is $v_0^{(k)}$ and let us call $u_1^{(k)},...,u_{d-1}^{(k)}$ the other neighbors (which are separate from the $v^{(k)}$'s). Note that, all these vertices are not random (their definition only depends on $k$).\\

Let us describe ways to exit the hypercube $h(\omega)$ that will provide lower bounds on quantities of the type $\widetilde{Q}_{0,x_0(\omega)+x}^{h(\omega)}$ (see Definition~\ref{def_k}). First, we have to go out of the edge $\{v_0^{(k)},v_d^{(k)}\}=\{0,e_k\}$ (recall that $p^\omega(0,e_k)\ge\delta$ on $A_k$), using one of the vector that points out of this edge. There are two possibilities
\begin{enumerate}
\item if this vector makes us exit $h(\omega)$, we have reached our goal (exiting $h(\omega)$),
\item if  this vector leads us to another point of the hypercube, we just go on the same direction for one more step, which makes us exit the hypercube $h(\omega)$ (see Figure \ref{dessin_E}).
\end{enumerate}

\begin{figure}[h]
   \includegraphics{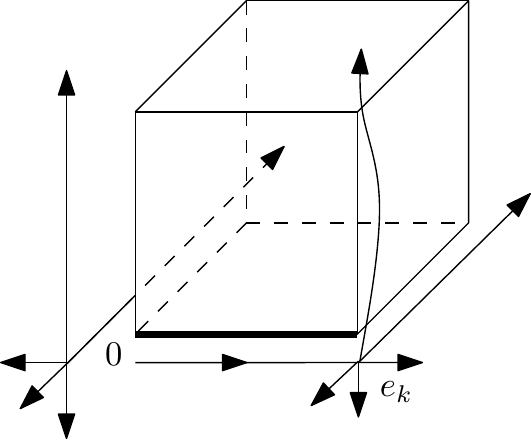}
      \caption{\label{dessin_E} The arrows represent the different strategies for the walker to exit the hypercube efficiently under condition $(E')_1$, starting at $0$. On the event $A_k$, the bold edge can be crossed from $0$ to $e_k$ with lower bounded probability, the other arrows may be hard to cross individually but as a group they provide a sufficient accessible escape route.}

\end{figure}

 There are many more ways to exit the hypercube but we will not need them, indeed our lower bounds on the quantities of the type $\widetilde{Q}_{0,x_0(\omega)+x}^{h(\omega)}$ (see Definition~\ref{def_k}) will be sufficient.

This intuition will guide us in our choice of $\alpha_x$. We labeled $2d$ vertices of $h(\omega)$ among $2^d$. For any point $x\in h(\omega)$ such that $x\notin\{v_0^{(k)},...,v_d^{(k)},u_1^{(k)},...,u_{d-1}^{(k)}\}$, we just choose the mark $\alpha_{x-x_0(\omega)}(\omega)=0$, where $x_0(\omega)$ is defined is Section~\ref{sect_markhyp}.\\

Recalling the definition of $\gamma_x$ at~(\ref{res1}), let us define
\begin{align*}
\alpha_{v_0^{(k)}-x_0(\omega)}(\omega)&=\gamma_{v_0^{(k)}-x_0(\omega)},\\
\alpha_{v_d^{(k)}-x_0(\omega)}(\omega)&=\gamma_{v_d^{(k)}-x_0(\omega)},
\end{align*}
as well as, for all $i\in\{1,...,d-1\}$,
\begin{align*}
\alpha_{v_i^{(k)}-x_0(\omega)}(\omega)&=\phi(v_i^{(k)}-v_0^{(k)})<\gamma_{v_i^{(k)}-x_0(\omega)},\\
\alpha_{u_i^{(k)}-x_0(\omega)}(\omega)&=\phi(u_i^{(k)}-v_d^{(k)})<\gamma_{u_i^{(k)}-x_0(\omega)}.
\end{align*}
Notice that, if we set $u_0^{(k)}=v_d^{(k)}(=e_k)$,
\begin{align*}
\sum_{i=0}^{d-1} \alpha_{v_i^{(k)}-x_0(\omega)}(\omega)&=\sum_{e\in U}\phi(e)-\phi(e_k),\\
\sum_{i=0}^{d-1} \alpha_{u_i^{(k)}-x_0(\omega)}(\omega)&=\sum_{e\in U}\phi(e)-\phi(-e_k).
\end{align*}

Hence, by~(\ref{eps}), there exists $\epsilon>0$ such that for any $k\in\{1,...,2d\}$, on the event $A_k$ ,
\begin{align}\label{res2}
\sum_{x\in \mathfrak{H}} (\gamma_x \wedge \alpha_x(\omega))=2\sum_{e\in U}\phi(e)-\left(\phi(e_k)+\phi(-e_k)\right)> 1+\epsilon.
\end{align}

\vspace{0.5cm}

{\it Second condition for a Markovian hypercube}

\vspace{0.5cm}

We have described how to obtain lower bounds on $\widetilde{Q}_{0,x_0(\omega)+x}^{h(\omega)}$ for $x\in \mathfrak{H}$ in the paragraph describing the intuition behind our choice of $\alpha_x$. Thus, on the event $A_k$ for any $k\in\{1,...,2d\}$, using that $p^\omega(0,e_k)=p^\omega(0,v_d^{(k)}-v_0^{(k)})\ge\delta$, and recalling that $u_0^{(k)}=v_d^{(k)}=e_k$, we can see that
\begin{align*}
\widetilde{Q}_{0,v_0^{(k)}}^{h(\omega)} &\ge Q_{v_0}^{h(\omega)},\\
\widetilde{Q}_{0,v_d^{(k)}}^{h(\omega)}&\ge\delta Q_{v_d}^{h(\omega)},\\
\widetilde{Q}_{0,v_i^{(k)}}^{h(\omega)}&\ge p^\omega(v_0^{(k)},v_i^{(k)}-v_0^{(k)})p^\omega(v_i^{(k)},v_i^{(k)}-v_0^{(k)})) \text{, for  } i\in\{1,...,d-1\},\\
\widetilde{Q}_{0,u_i^{(k)}}^{h(\omega)}&\ge\delta p^\omega(u_0^{(k)},u_i^{(k)}-u_0^{(k)})p^\omega(u_i^{(k)},u_i^{(k)}-u_0^{(k)}))\text{, for } i\in\{1,...,d-1\}.
\end{align*}
Recalling that $\alpha_{x}(\omega)=0$ as soon as $(x_0(\omega)+x)\notin\{v_0^{(k)},...,v_d^{(k)},u_1^{(k)},...,u_{d-1}^{(k)}\}$, we deduce by regrouping the terms properly that:
\begin{align*}
\prod_{x\in \mathfrak{H}} \Bigl(\widetilde{Q}_{0,x_0(\omega)+x}^{h(\omega)}\Bigr)^{-\alpha_x(\omega)}&\le \delta^{-\sum_{e\in U:e\neq -e_k}\phi(e)}\prod_{e\in U, e\neq e_k} \left(p^\omega(v_0^{(k)},e)\right)^{-\phi(e)}\\
&\times\prod_{e\in U, e\neq -e_k} \left(p^\omega(u_0^{(k)},e)\right)^{-\phi(e)}\\
& \times \prod_{i=1}^{d-1}\left(p^\omega(v_i^{(k)},v_i^{(k)}-v_0^{(k)}))\right)^{-\phi(v_i^{(k)}-v_0^{(k)})}\\
& \times \prod_{i=1}^{d-1}\left(p^\omega(u_i^{(k)},u_i^{(k)}-u_0^{(k)}))\right)^{-\phi(u_i^{(k)}-u_0^{(k)})}.
\end{align*}

We can notice on the right-hand side of the previous equations we have ${\bf P}$-independence between the terms
\begin{enumerate}
\item $\prod_{e\in U, e\neq e_k} \left(p^\omega(v_0^{(k)},e)\right)^{-\phi(e)}$,
\item $\prod_{e\in U, e\neq -e_k} \left(p^\omega(u_0^{(k)},e)\right)^{-\phi(e)}$,
\item $p^\omega(v_i^{(k)},v_i^{(k)}-v_0^{(k)})$, for all $i\in [1,d-1]$,
\item $p^\omega(u_i^{(k)},u_i^{(k)}-u_0^{(k)})$, for all $i\in [1,d-1]$.
\end{enumerate}

Hence, for any $k\in\{1,...,2d\}$, the annealed expectation of this previous quantity is finite, by~(\ref{zzz}). Thus, using translation invariance, that
\begin{align*}
{\bf E}\left[\prod_{x\in \mathfrak{H}} \Bigl(\widetilde{Q}_{0,x_0(\omega)+x}^{h(\omega)}\Bigr)^{-\alpha_x(\omega)}\right]&= {\bf E}\left[\sum_{k=1}^{2d}\1{A_k}\prod_{x\in \mathfrak{H}}\Bigl(\widetilde{Q}_{0,x_0(\omega)+x}^{h(\omega)}\Bigr)^{-\alpha_x(\omega)}\right]\\
& \le 2d\delta^{-\sum_{e}\phi(e)} \prod_{k=1}^{2d}\left({\bf E}\left[\prod_{e\neq e_k} p^\omega(0,e)^{-\phi(e)}\right]\right)^{2d}\\
&<\infty.
\end{align*}

This concludes the proof, together with \eqref{res1} and \eqref{res2} since all the three parts of the definition of $(K)_1$ are verified.

 \end{proof}

{\bf An example that verifies $(K)_1$ but no former criteria for ballistic behavior}

We are now going to introduce an example in which we can verify $(K)_1$ but not $(E')_1$, showing that the ellipticity criterion  $(K)_1$ is more general. This example also satisfies condition $(E)_0$, and we will also prove that it verifies condition $(T)$ and thus directional transience.

Let us choose $T$ a random variable such that $2d+1\le T<\infty$ ${\mathbf P}$-a.s.,  $\ES[T^{\frac{1}{4d}}]=\infty$ and $\ES[T^{\frac{1}{8d}}]<\infty$. Furthermore, we introduce  an independent random variable $i_0$ that is uniform on $\{1,...,2d\}$.\\

Let us now define $p^\omega(0,\cdot)$ in terms of $T$ and $i_0$ as this will give us the transition probabilities for this walk. Let $\epsilon\in (\frac{1}{2d+1},\frac{2d}{2d+1})$  and set:
\begin{displaymath}
p^\omega(0,e_i)=\left\{
\begin{array}{ll}
\frac{1}{T} & \textrm{if }i=i_0,\\ 
\frac{1-\epsilon-\frac{\1{1\le i_0\le d}}{T}}{d-\1{1\le i_0\le d}}& \textrm{if } i\in\{1,...,d\}\setminus\{ i_0\},\\
\frac{\epsilon-\frac{\1{d+1 \le i_0\le 2d}}{T}}{d-\1{d+1\le i_0\le 2d}}& \textrm{if } i\in\{d+1,...,2d\}\setminus\{ i_0\}.
\end{array} \right.
\end{displaymath}

Let us denote ${\mathbf P}^{\text{expl}}[\cdot]$ the law of this environment. 
\begin{proposition}\label{new_example}
The environment ${\bf P}^{\text{expl}}[\cdot]$ verifies $(K)_1$ but does not verify $(E')_1$. Furthermore a RWRE in an environment given by ${\mathbf P}^{\text{expl}}[\cdot]$ verifies condition $(T)$ and $(E)_0$.
\end{proposition}

\begin{proof}
Note that $\ES[T^{\frac{1}{8d}}]<\infty$ ensures that $(E)_0$ holds. Let us prove that the walk is directionally transient and verifies conditions $(T')$.\\

The transition probabilities of ${\bf P}^{\text{expl}}[\cdot]$ are such that the walk has a strong drift toward $\ell_0=e_1+...+e_d$, as soon as $\epsilon$ is small enough. Indeed, we have
\[
P^\omega_0\left[X_1\cdot \ell_0=1\right]=1-P^\omega_0\left[X_1\cdot \ell_0=-1\right]=1-\epsilon \qquad {\bf P}-\text{-a.s.}
\]
Thus, the process $(X_n\cdot l_0)_n$ is a random walk on $\mathbb{Z}$ in a deterministic environment such that it performs a jump toward the right with probability $1-\epsilon$. Therefore, it is now clear that the walk $X$ is transient towards $\ell_0$ and
\[
\frac{X_n\cdot \ell_0}{n}\longrightarrow 1-2\epsilon.
\]

\vspace{0.5cm}

{\it Verifying condition $(T)$}

\vspace{0.5cm}

We want to prove condition $(T)^{l_0}$ (see~\eqref{hyp_T}), this means we need a neighborhood of $\ell_0$. For this we consider $\ell_0'=\ell_0+\sum_{i=1}^{d}\epsilon_ie_i$, where, for all $i\in\{1,...,d\}$, $\epsilon_i\in(-\epsilon,\epsilon)$.  Notice that, for all $n\in\mathbb{N}$, $X_n\cdot \ell_0'=X_n\cdot\ell_0+\sum_{i=1}^d \epsilon_i (X_n\cdot e_i)$. Thus, obviously we have $X_n\cdot \ell_0'\ge X_n\cdot \ell_0-\epsilon n$.

 Now, using a standard large deviation type argument, we can show that for some $\lambda>0$ small enough, we have for all $M>0$
\[
\PR\left[X_n\cdot \ell_0'<-M\right] \le C e^{-\lambda (n+M)},
\]
to prove this inequality we use the fact that $\epsilon <1/5$. Furthermore, we get:
\begin{align*}
\PR_0\left[T_{-M}^{(X_n\cdot \ell_0')}<T_a^{(X_n\cdot \ell_0')}\right]&\le \sum_{n=0}^\infty \PR\left[X_n\cdot \ell_0'<-M\right]\le C\exp(-\lambda M),
\end{align*}
where for $x\in \Z$ we denote $T_x^{(X_n\cdot \ell_0')}$ denotes the first hitting time of $x$ by the walk $X_n\cdot \ell_0'$.

We can now check easily that condition $(T)$ is verified by choosing $a=L\ge1$ and $M=bL$, for some $b>0$, since
\[
\PR\left[T_{-bL}^{(X_n\cdot \ell_0)}<T_L^{(X_n\cdot \ell_0)}\right]\leq Ce^{-cL}.
\]

\vspace{0.5cm}

{\it Verifying that condition $(E')_1$ is not satisfied}

\vspace{0.5cm}

Let us prove that the condition $(E')_1$ is not satisfied. Indeed, for any family of real numbers $\{\phi(e), e\in U\}\in (0,\infty)^{2d}$ such that
\[
2\sum_{e\in U} \phi(e)-\sup_{e\in U} (\phi(e)+\phi(-e))>1,
\]
there exists $e_0\in U$ such that $\phi(e_0)\ge1/4d$. Then, we have
\begin{align*}
\ES\Bigl[\exp\bigl( \phi(e_0)\log \frac 1{p^{\omega}(0,e_0)}\bigr)\Bigr]&\ge \ES\left[\1{i_0=e_0}T^{\phi(e_0)}\right]\\
&\ge \frac{1}{2d} \ES\left[T^{\frac{1}{4d}}\right]=\infty,
\end{align*}
in particular this implies that $(E')_1$ is not satisfied.

\vspace{0.5cm}

{\it Verifying that condition $(K)_1$ is  satisfied}

\vspace{0.5cm}

On the other hand, let us prove that condition $(K)_1$ is verified. By the definition of ${\bf P}^{\text{expl}}[\cdot]$, for any $x\in\mathfrak{H}$, $Q_x^{\mathfrak{H}}\ge \epsilon/d$ ${\bf P}$-a.s.~and thus has any moments. At this point one could conclude the proof by using Remark~\ref{suff_cond}. 

For illustrative purposes, we shall also verify $(K)_{\alpha}$ for any given $\alpha>0$. We can choose, for all $x\in\mathfrak{H}$, $\gamma_x=\alpha+1$ for instance in order to verify the first property of $(K)_{\alpha}$. Then, we define a Markovian hypercube such that $h(\omega)=\mathfrak{H}$ ${\bf P}$-a.s.~so that it is in fact deterministic and $x_0(\omega)=0$ ${\bf P}$-a.s. Then, we choose $\alpha_0(\omega)=\alpha+1$ ${\bf P}$-a.s.~and, for any $x\in\mathfrak{H}\setminus\{0\}$, we fix $\alpha_x(\omega)=0$ $\PR$-a.s. This implies that
\[
{\bf E}\Bigl[\prod_{x\in \mathfrak{H}} \Bigl(\widetilde{Q}_{0,x_0(\omega)+x}^{h(\omega)}\Bigr)^{-\alpha_x(\omega)}\Bigr]\le {\bf E}\Bigl[\Bigl(Q_{0}^{\mathfrak{H}}\Bigr)^{-(\alpha+1)}\Bigr]<\infty,
\]
and we conclude with
\[
\sum_{x\in \mathfrak{H}} (\gamma_x \wedge \alpha_x(\omega))=\alpha+1.
\]

This means that $(K)_{\alpha}$ is verified for any $\alpha>0$ (with $\epsilon=1$).
\end{proof}

\subsection{Why do unit hypercubes appear?}\label{sect_hypercube}

Let us explain informally why traps can only exists if the walk can get trapped inside a hypercube. Intuitively, if there is a finite shape $\mathcal{S}$ in which the walk stays trapped, then every edge getting out of this edge has an abnormally small probability of being crossed (making this edge a rare one). If the \lq\lq corners\rq\rq~ of that shape were translated onto the hypercube $\mathfrak{H}$, using the i.i.d. character of the environment, we could create a trap inside $\mathfrak{H}$ (see Figure~\ref{dessin_shape}). This trap inside $\mathfrak{H}$ should typically be more likely to appear than the initial trap in $\mathcal{S}$ since we have diminished the number of atypically \lq\lq hard-to-cross\rq\rq~(thus rare) edges.

\begin{figure}[h]

   \includegraphics{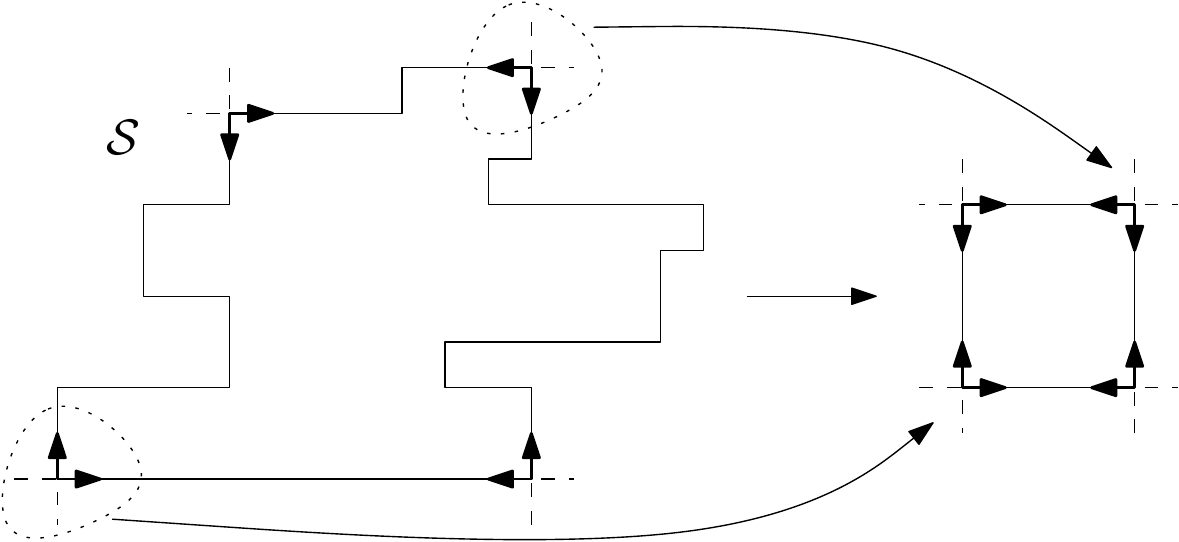}
   \caption{\label{dessin_shape} The ``corners'' of a shape $\mathcal{S}$ in $\mathbb{Z}^2$ translated onto a hypercube.}
\end{figure}

Conversely, we should also show that there are RWREs in elliptic i.i.d.~environments in $\Z^d$ with $d\geq 2$ that have zero speed but cannot be trapped on only one edge.\\

Let us choose $T$ a positive random variable verifying that ${\bf P}[T\leq \frac 1 {2}]=1$ and ${\bf P}[T^{-1}\geq n]\geq cn^{-1/2^d}$. Furthermore, we introduce an independent random variable $B_0$ which is uniform on the set of orthonormal basis of $\Z^d$ (i.e. $B_0$ is some orthonormal basis). 

We are going to define $p^{\omega}(0,\cdot)$ in terms of $T$ and $B_0$ as this will give us the law of the transition probabilities for our walk. We set
\begin{equation}\label{porb1}
p^{\omega}(0,e_i)=\begin{cases} \frac T {d} & \text{ if $e_i\in B_0$,} \\ \frac 1{d} -  \frac T {d}  & \text{ if $e_i\notin B_0$}. \end{cases} 
\end{equation}

In this model, we typically imagine that $T$ is small which means that the edges in $B_0$ are hard to cross whereas the others are not. It is obvious that the walk cannot get trapped on a single edge. Indeed, from every point there are at least $d$ edges which have probability at least $1/(2d)$ of being taken. 

On the other hand, the exit time of a hypercube has infinite expectation.
Indeed, let us introduce the random variables $B_0^{(x)}$, for all $x\in\mathbb{Z}^d$, which are distributed like $B_0$ such that they are all independent from each other. Similarly we introduce $T_0^{(x)}$, for any $x\in\mathbb{Z}^d$.

For $x\in \mathfrak{H}$, let us consider the event $\{B_0^{(x)}=\partial_x \mathfrak{H} ,\text{ for all $x\in \mathfrak{H}$}\}$. This event has a positive probability ($(2^{-d})^{2^d}$). On this event, from any point $x\in \mathfrak{H}$, we know that $P^{\omega}_x[X_1\notin \mathfrak{H}]\leq \max_{x\in \mathfrak{H}} \frac {T_0^{(x)}}d$. This implies that 
\[
\ES[T^{\text{ex}}_{\mathfrak{H}}] \geq c{\bf E}\Bigl[\text{Geom}\Bigl( \max_{x\in \mathfrak{H}} \frac {T_0^{(x)}}d\Bigr)\Bigr] =c {\bf E}\Bigl[\min_{x\in \mathfrak{H}} \frac 1{T_0^{(x)}}\Bigr].
\]

To see that the right-hand side is infinite, we simply compute the tail of $\min_{x\in \mathfrak{H}} T^{(x)}$,
\begin{equation}\label{non_integ_ex}
{\bf P}\Bigl[\min_{x\in \mathfrak{H}} \frac 1{T_0^{(x)}}\geq n\Bigr]={\bf P}[T^{-1}\geq n]^{2^d}\geq cn^{-1},
\end{equation}
which is non-integrable. This means that $\ES[T^{\text{ex}}_{\mathfrak{H}}] =\infty$. This indicates that trapping occurs. If the walk was directionally transient, then we could use Theorem~\ref{lb_tail} and prove that the walk is sub-ballistic, even though one single edge is not enough to trap a walk. Nevertheless, since the transition probabilities are symmetric, the walk is not directionally transient.

In order to address this issue, let us introduce the following similar model. Recall the notation $T$ and $B_0$ defined above~(\ref{porb1}). We will now define a RWRE in $\Z^{d+1}$. For this, let us point out that $B_0$ is a $d$-dimensional basis in a $(d+1)$-dimensional space such that  a.s.~$B_0\cap\{e_{d+1},e_{2(d+1)}\}=\emptyset$. Moreover, we have that a.s.~$\{e \in U,\ e\in B_0 \text{ or } -e\in B_0\}=U\setminus\{e_{d+1},e_{2(d+1)}\}$, where $U$ is the set of the $2(d+1)$ unit vectors of of $\mathbb{Z}^{d+1}$. After noticing this, we can set
\begin{equation}\label{porb1*}
q^{\omega}(0,e_i)=\begin{cases} \frac T {C(T,d)} & \text{ if $e_i\in B_0$,} \\ \frac 1{C(T,d)} -  \frac T {C(T,d)}  & \text{ if $-e_i\in B_0$,} \\    \frac {2T}{C(T,d)} &\text{ if $i=1+d$,}\\ \frac T{C(T,d)} &\text{ if $i=2(1+d)$.}\end{cases} 
\end{equation}
where $C(T,d)$ is a normalizing constant so that $q(0,\cdot)$ yields a probability transition. Since $T\leq 1/2$ a.s.~an elementary computation shows that $d\leq C(T,d)\leq 2(d+1)$.

Let us call ${\bf Q}^{\text{ex}}[\cdot]$ the law of the i.i.d.~environment arising from the previous construction.

\begin{figure}[h]

   \includegraphics{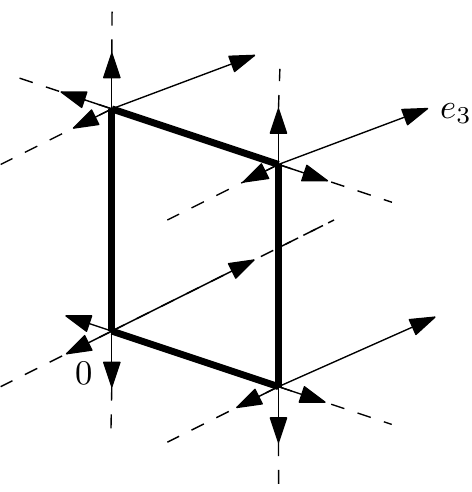}
   \caption{\label{dessin_trap} Transition probability around a square $\mathfrak{S}$ in $\mathbb{Z}^3$ on the event $\{B_0^{(x)}=\partial_x\mathfrak{S}\setminus\{e_{3},e_{6}\}\text{, for all } x\in\mathfrak{S}\}$. Bold edges are crossed with a lower bounded probability. Direction $e_3$ is preferred among those leading out of $\mathfrak{S}$.}
\end{figure}

\begin{proposition}\label{ex_nul}
Let $X_n$ be the RWRE in the environment given by ${\bf Q}^{\text{ex}}[\cdot]$. For any $d\geq 1$, it verifies that 
\begin{enumerate}
\item $X_n$ is transient in the direction $e_{1+d}$.
\item $X_n$ has zero velocity.
\end{enumerate}

If furthermore $d\geq 2$, then the walk $X_n$ is unlikely to localize on one edge in the sense that there exists $C(d)<\infty$ such that
\[
\lim_{n\to \infty} \PR[\text{there exists $i\in [0,n]$, such that }\abs{\{X_j,\ j\in[i,i+C(d)\ln n]\}} =2] = 0.
\]
\end{proposition}

\begin{proof}
The directional transience follows immediately from the fact that $X_n\cdot e_{1+d}$ is a time-changed 2-biased random walk on $\Z$.

A computation similar to~(\ref{non_integ_ex}), proves that the annealed exit time of a hypercube is infinite. The directional transience and Theorem~\ref{lb_tail} imply that the asymptotic velocity is $0$.

We know that from every point there are at least $d$ edges which have probability at least $1/(4(d+1))$ of being taken. If $d\geq 2$, it can then easily be show that the time spend on one edge is stochastically upper-bounded by a geometric random variable or parameter $1/(4(d+1))$. The final point of our proposition therefore follows from a simple union bound.
\end{proof}

\begin{remark}
In fact, this walk also verifies condition $(P)_M^{e_{d+1}}$ for some $M\ge15d+5$, see Theorem $4$ of \cite{BRS}.
\end{remark}


\section{Proof of Theorem \ref{main} and Theorem \ref{cor_TCL}}\label{sect_thm_proof}


This section is dedicated to the proof of Theorem \ref{main} and Theorem \ref{cor_TCL} which state respectively the positive speed and central limit theorems for transient random walks in an elliptic i.i.d. environment satisfying $(E)_0$ (see \eqref{def_e}), the polynomial condition (see \eqref{def_pm}) and the ellipticity conditions $(K)_{\alpha}$, defined by Definition \ref{def_k}.

The following proposition gives an estimate on the tail of the regeneration time $\tau_1$, defined in Section \ref{sect_regen}.

\begin{proposition}\label{tailtau}
Let $\ell\in S^{d-1}$, $\alpha>0$ and $M\ge 15d+5$. Assume that $(P)_M^\ell$ is satisfied and that the ellipticity conditions $(E)_0$ and $(K)_\alpha$ hold (resp. defined in \eqref{def_e} and \eqref{def_k}). Then, there exists $\delta>0$ such that
\[
 \PR[\tau_1>u]\le Cu^{-(\alpha+\delta)}.
\]
\end{proposition}

On the one hand, this Proposition implies positive speed (see Theorem \ref{main}) by using Theorem \ref{speed_regen}. On the other hand, it also implies the central limit theorems (see Theorem \ref{cor_TCL}), by using Theorem \ref{speed_regen} for the annealed case and the main result of Bouchet, Sabot and dos Santos \cite{BDSS} for the quenched case. Note that, for the latter, we need to prove that condition $(T)'$ is verified under our assumptions: as we have stated in Section \ref{sect_questions},  it has been shown in~\cite{CR} that under $(E)_0$ the polynomial condition $(P)_M$ is equivalent to $(T)'$. We properly state this result in Section~\ref{sect_ref_res} (see Theorem~\ref{PequivT}).\\

The goal is thus to prove Proposition \ref{tailtau}, which is done in Section \ref{sect_tailtau}.
The proof of Proposition \ref{tailtau} relies on two results of \cite{CR} which state that, under $(E)_0$, the polynomial condition is equivalent to condition $(T')$ (see Theorem \ref{PequivT}) and that some atypical quenched exit estimates hold (see Proposition \ref{atypical}).

Another key of the proof is that, under $(K)_1$, with great probability, the walker reaches some point sufficiently away with sufficiently large quenched probability. The exact meaning of this sentence will be clarified in Proposition \ref{reachlog} and Corollary \ref{reachlog2} in the following Section.\\
These three arguments allow us to derive an estimate on the tail of the regeneration time $\tau_1$, using arguments similar to those used by Sznitman in \cite{Szn01}.


\subsection{Attainability estimates}\label{sect_att}


Let us prove the following result which is needed for the proof of Theorem \ref{main} and Theorem~\ref{cor_TCL}. For this purpose, let us define the $2^d$ following paths starting at $0$ and reaching a point at distance $n$ by visiting at most $n+2^d$ vertices, and without coming back to $0$. We are going to use the marked Markovian hypercube $(h(\omega),(\alpha_x(\omega))_{x\in\mathfrak{H}})$ associated to condition $(K)_\alpha$ and his particular corner $x_0(\omega)$, defined in \eqref{def_x0}.

As we will see, these \emph{paths} are not really paths but rather unions of trajectories. Indeed, we construct these objects in two steps. The first step consists in (starting at $0$ and without coming back to $0$) going out of the Markovian hypercube reaching some neighbor $y$ of the hypercube. This does not define one path but a union of paths. For the second step, we construct an actual path that starts in $y$ and goes more or less straight away from $0$, without intersecting itself (see Figure~\ref{dessin_paths}). We allow ourselves to misname these objects and call them \emph{paths} as the important point is that they go from $0$ to some point that is far away without coming back to $0$. Moreover, even though we do not control the number of steps in these paths, we can upper bound the number of different points that they visit.\\

For any $x\in \mathfrak{H}$ and $n\in\mathbb{N}$, let  $\mathcal{Y}^{(n)}_x=(y_0^{(x)},...,y_n^{(x)})$ be the path constructed as follows:
\begin{enumerate}
\item the path starts at $y_0^{(x)}=0$;
\item\label{her} the path goes out of the marked Markovian hypercube $(h(\omega),(\alpha_z(\omega))_{z\in\mathfrak{H}})$, without coming back to $0$, and via a neighbour $y_1^{(x)}$ of $x_0(\omega)+x$ such that
\[
P^\omega_0\left[T_{\partial h(\omega)}< T_0^+, X_{T_{\partial h(\omega)}}=y_1^{(x)}\right]=\max_{y\in\partial_{x_0(\omega)+x} h(\omega)} P^\omega_0\left[T_{\partial h(\omega)}< T_0^+, X_{T_{\partial h(\omega)}}=y\right],
\]
and $y_1^{(x)}$ is chosen arbitrarily if several vertices realize this last equality;
\item the rest of the path is a nearest-neighbour path such that, for all $i\in\{1,...n-1\}$, $y_{i+1}^{(x)}$ is a neighbour of $y_i^{(x)}$ such that
\[
p^\omega(y_i^{(x)},y_{i+1}^{(x)}-y_i^{(x)})=\max_{e\in U: y_i^{(x)}+e\notin \mathfrak{H}_{y_i^{(x)}-x}}p^\omega(y_i^{(x)},e)= Q_{y_i^{(x)}}^{\mathfrak{H}_{y_i^{(x)}-x}},
\]
where $Q$ is defined in \eqref{def_Q},  $\mathfrak{H}_{y_i^{(x)}-x}$ is defined in \eqref{def_hx}, and $y_{i+1}^{(x)}$ is chosen arbitrarily if several vertices realize this last equality.
\end{enumerate}

See Figure \ref{dessin_paths} for a scheme of this construction. Note also that, for all $x\in\mathfrak{H}$, $y_0^{(x)}$ and $y_1^{(x)}$ are not necessarily neighbors.

Essentially, the path $\mathcal{Y}^{(n)}_x$ first goes out of the Markovian hypercube using the corner $x_0(\omega)+x$ and then continues in the same global direction for $n-1$ steps (see Figure \ref{dessin_paths}), such that it goes further from $0$ at each step, using the same orthonormal basis that points out of $h(w)$ from $x_0(\omega)+x$. Hence, for all $x\in\mathfrak{H}$, $|y_n^{(x)}|_1\ge n$.

\begin{figure}[h]
   \includegraphics{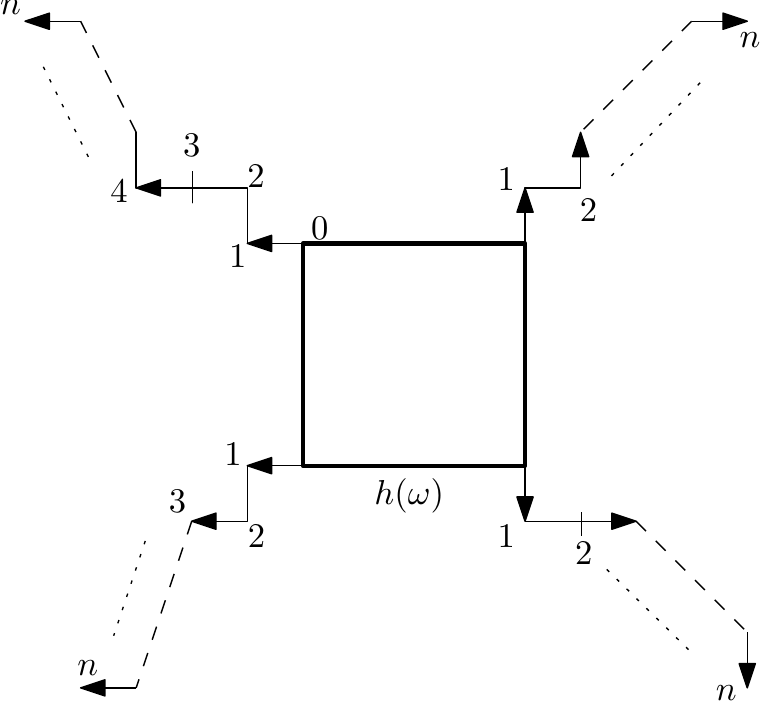}
      \caption{\label{dessin_paths} Chosen ``paths'' to reach a point at distance $n$.}

\end{figure}

Notice that once the paths $(\mathcal{Y}^{(n)}_x)_{x\in\mathfrak{H}}$, are out of the Markovian hypercube, they do not intersect. That fact is very helpful in the computations of the following Proposition \ref{reachlog}.\\

Even though $\mathcal{Y}^{(n)}_x$ is not a path, we call  quenched probability of the path $\pi_x^{(n)}$ of $\mathcal{Y}^{(n)}_x$ the following quantity
\begin{align}\label{def_pix}
\pi_x^{(n)}&=P^\omega_0\left[T_{\partial h(\omega)}< T_0^+, X_{T_{\partial h(\omega)}}=y_1^{(x)}\right]\prod_{i=1}^{n-1} Q_{y_i}^{\mathfrak{H}_{y_i-x}},
\end{align}
dropping the superscript ``$(x)$'' of the $y_i$'s for simplicity.

The next proposition states that, with high $\mathbf{P}$-probability, one of the paths depicted in Figure~\ref{dessin_paths} has a decent chance of being followed.

\begin{proposition}\label{reachlog}
Consider a RWRE in an elliptic environment satisfying condition $(K)_\alpha$, $\alpha>0$.
For any $\delta>0$, there exists a constant $\eta>0$, such that, for any $u$ large enough, we have
\[
{\bf P}\left[ \max_{x\in\mathfrak{H}} \pi_x^{(\lfloor \eta\log(u)\rfloor)}< u^{-\frac{\alpha+2\delta}{\alpha+\epsilon}}\right]\le\frac{1}{u^{\alpha+\delta}},
\]
where $\lfloor\cdot\rfloor$ is the floor function, $\pi_x^{(n)}$ is defined by \eqref{def_pix} and $\epsilon$ comes from condition $(K)_\alpha$.
\end{proposition}

\begin{proof}
Let us emphasize again some facts about the paths $\mathcal{Y}_x^{(n)}$, $x\in\mathfrak{H}$, $n\in\mathbb{N}$, and their quenched probability $\pi_x^{(n)}$:
\begin{enumerate}
\item out of the Markovian hypercube $h(\omega)$, the paths do not intersect, i.e. 
\[
\bigcap_{x\in\mathfrak{H}} \left\{y_i^{(x)},i\in\{1,...,n\}\right\}=\emptyset;
\]
\item for any $i\in\{1,...n-1\}$, conditionnally on $h(w)$ and $y_0^{(x)},...,y_i^{(x)}$, for all $x\in\mathfrak{H}$, the quenched probabilities $p^\omega(y_i^{(x)},y_{i+1}^{(x)}-y_i^{(x)})$ are independent and are distributed like the random variable $Q_{x}^{\mathfrak{H}}$ defined in \eqref{def_Q}. These independence properties rely on the fact that $h(\omega)$ is a Markovian hypercube, in particular here we use Remark \ref{rem_markov}.

\item using the definition \eqref{def_Qtilde} of $\widetilde{Q}$ and the property \eqref{her} of the construction page \pageref{her}, we have that, for any $x\in\mathfrak{H}$,
\[
P^\omega_0\left[T_{\partial h(\omega)}< T_0^+, X_{T_{\partial h(\omega)}}=y_1^{(x)}\right]\ge \frac{1}{d} \widetilde{Q}_{0,x_0(\omega)+x}^{h(\omega)},
\]
thus, using \eqref{def_pix},
\[
\pi_x^{(n)}\ge \frac{1}{d} \widetilde{Q}_{0,x_0(\omega)+x}^{h(\omega)}\prod_{i=1}^{n-1} Q_{y_i}^{\mathfrak{H}_{y_i-x}}.
\]
\end{enumerate}

Now, fix $\eta>0$ and $\delta>0$ and recall Definition~\ref{def_k} of the marks $(\alpha_x(\omega))_{x\in\mathfrak{H}}$ of the marked Markovian hypercube $(h(\omega),(\alpha_x(\omega))_{x\in \mathfrak{H}})$. Using Markov inequality, the condition $(K)_\alpha$ and the previous remarks, we have, for all $u>\exp(1/\eta)$,
\begin{align*}
{\bf P}\left[\max_{x\in\mathfrak{H}} \pi_x^{(\lfloor \eta\log(u)\rfloor)}< u^{-\frac{\alpha+2\delta}{\alpha+\epsilon}}\right]&\le \frac{d^{\alpha+\epsilon}}{u^{\alpha+2\delta}} {\bf E}\left[\min_{x\in\mathfrak{H}} \left(d\pi_x^{(\lfloor \eta\log(u)\rfloor)}\right)^{-(\alpha+\epsilon)}\right]\\
&\le \frac{d^{\alpha+\epsilon}}{u^{\alpha+2\delta}} {\bf E}\left[ \prod_{x\in\mathfrak{H}}\left(d\pi_x^{(\lfloor \eta\log(u)\rfloor)}\right)^{-(\alpha_x(\omega)\wedge\gamma_x)}\right]\\
\le   \frac{d^{\alpha+\epsilon}}{u^{\alpha+2\delta}} {\bf E}&\left[ \prod_{x\in\mathfrak{H}}\left(\left(\widetilde{Q}_{0,x_0(\omega)+x}^{h(\omega)}\right)^{-\alpha_x(\omega)}\prod_{i=1}^{\lfloor \eta\log(u)\rfloor-1}\left( Q_{y_i}^{\mathfrak{H}_{y_i-x}}\right)^{-\gamma_x}\right)\right]\\
&\le   \frac{d^{\alpha+\epsilon}}{u^{\alpha+2\delta}} {\bf E}\left[ \prod_{x\in\mathfrak{H}}\left(\widetilde{Q}_{0,x_0(\omega)+x}^{h(\omega)}\right)^{-\alpha_x(\omega)}\right]\\
&\qquad\qquad\times\prod_{x\in\mathfrak{H}}\left[{\bf E}\left[\left( Q_{x}^{\mathfrak{H}}\right)^{-\gamma_x}\right]\right]^{\lfloor \eta\log(u)\rfloor-1},
\end{align*}
where we used all previously mentioned independence properties. Introducing
\[
C=\max \left(\mathbf{E}\left[\prod_{x\in\mathfrak{H}}\left(\widetilde{Q}_{0,x_0(\omega)+x}^{h(\omega)}\right)^{-\alpha_x(\omega)}\right]; \mathbf{E}\left[\left( Q_{x}^{\mathfrak{H}}\right)^{-\gamma_x}\right], x\in\mathfrak{H}\right)<\infty.
\]
which is finite thanks to condition $(K)_\alpha$, we can see that
\[
{\bf P}\left[\max_{x\in\mathfrak{H}} \pi_x^{(\lfloor \eta\log(u)\rfloor)}< u^{-\frac{\alpha+2\delta}{\alpha+\epsilon}}\right]\le \frac{d^{\alpha+\epsilon}C^{\eta\log(u)}}{u^{\alpha+2\delta}}.
\]

Finally, for $\eta>0$ small enough and $u$ large enough,
\[
\mathbf{P}\left[\max_{x\in\mathfrak{H}} \pi_x^{(\lfloor \eta\log(u)\rfloor)}< u^{-\frac{\alpha+2\delta}{\alpha+\epsilon}}\right]\le \frac{1}{u^{\alpha+\delta}}.
\]

\end{proof}

The proof of the following consequence is straightforward.

\begin{corollary}\label{reachlog2}
Consider a RWRE in an elliptic environment satisfying condition $(K)_\alpha$, $\alpha>0$.
For any $\delta>0$, there exists a constant $\eta>0$, such that, for any $u$ large enough, we have
\[
\mathbf{P}\left[ \max_{y:|y|_1=\lfloor \eta\log(u)\rfloor} P^\omega_0\left[T_y<T^+_0\right]< u^{-\frac{\alpha+2\delta}{\alpha+\epsilon}}\right]\le\frac{1}{u^{\alpha+\delta}},
\]
where $\lfloor\cdot\rfloor$ is the floor function and $\epsilon>0$ comes from condition $(K)_\alpha$.
\end{corollary}


\subsection{Polynomial condition and atypical quenched exit estimates}\label{sect_ref_res}


In this section, we just recall two results previously obtained by Campos and Ram\'irez in \cite{CR}. This uses conditions $(E)_0$, $(P)_M^\ell $ and $(T')$ defined respectively at~\eqref{def_e}, \eqref{def_pm} and~\eqref{hyp_T}.

This first theorem states that, under some light assumptions, the polynomial condition implies condition $(T')$.

\begin{theorem}[Theorem 1.1 of Campos and Ram\'irez, \cite{CR}]\label{PequivT}
Consider a random walk in an i.i.d. environment in dimensions $d\ge2$. Let $\ell\in S^{d-1}$ and $M\ge 15d+5$. Assume that the environment satisfies the ellipticity condition $(E)_0$. Then the polynomial condition $(P)_M^\ell$ is equivalent to $(T')^\ell$.
\end{theorem}

The following proposition allow us to compute atypical quenched exit estimates. Before stating the results, we need some definitions, similar to those introduced in \cite{CR,BRS}.

Let us consider a RWRE in an elliptic i.i.d. environment and $\ell\in S^{d-1}$, then $(P)_M^\ell$ implies that the walk has an asymptotic direction (see \cite{FS}), i.e. the following limit exists:
\begin{align}\label{def_v}
\hat{v}=\lim_{n\rightarrow\infty} \frac{X_n}{|X_n|_2}.
\end{align}

There exists $i_0\in [1,2d]$ such  that $\hat{v}\cdot e_{i_0}>0$. Assume also that $e_{i_0}$ is the vector of the canonical basis which is the nearest of $\hat{v}$, so that the angle between $v$ and $e_{i_0}$ is upper-bounded and we have
\[
\hat{v}\cdot e_{i_0}\ge\frac{1}{\sqrt{d}}.
\]

Moreover, for any $z\in\mathbb{Z}^d$, let $P(z)$ be the projection of $z$ on $v$ along the hyperplane $H=\{x\in\mathbb{R}^d:x\cdot e_{i_0}=0\}$, defined by
\[
P(z)=\left(\frac{z\cdot e_{i_0}}{\hat{v}\cdot e_{i_0}}\right)\hat{v},
\]
and let $Q(z)$ be the projection of $z$ on $H$ along $\hat{v}$ so that
\[
Q(z)=z-P(z).
\]

For $x\in\mathbb{Z}^d$, $\beta>0$ and $L>0$, we define the \emph{tilted boxes} with respect to the asymptotic direction $\hat{v}$ by:
\begin{align}\label{def_Btilted}
B_{\beta,L}(x)=\left\{y\in\mathbb{Z}^d: -L^\beta<(y-x)\cdot e_{i_0}<L, ||Q(y-x)||_\infty<L^\beta\right\},
\end{align}
and their front boundary by
\begin{align}\label{def_front}
\partial^+B_{\beta,L}(x)=\left\{y\in\partial B_{\beta,L}(x): (y-x)\cdot e_{i_0}=L\right\}.
\end{align}

\begin{remark}\label{elem_geom}
An elementary geometric computation (see Figure \ref{dessin_tiltedbox}) shows for $x_1,x_2\in \Z^d$ such that $(x_1-x_2)\cdot \hat v \ge0$ and $x_2\in B_{\beta,L}(x_1)$, we have $\abs{\abs{x_1-x_2}}_{\infty}\leq (1+\sqrt d)L^{\beta}$ since $\hat{v}\cdot e_{i_0}\ge\frac{1}{\sqrt{d}}$.
\end{remark}

\begin{figure}[h]
   \includegraphics{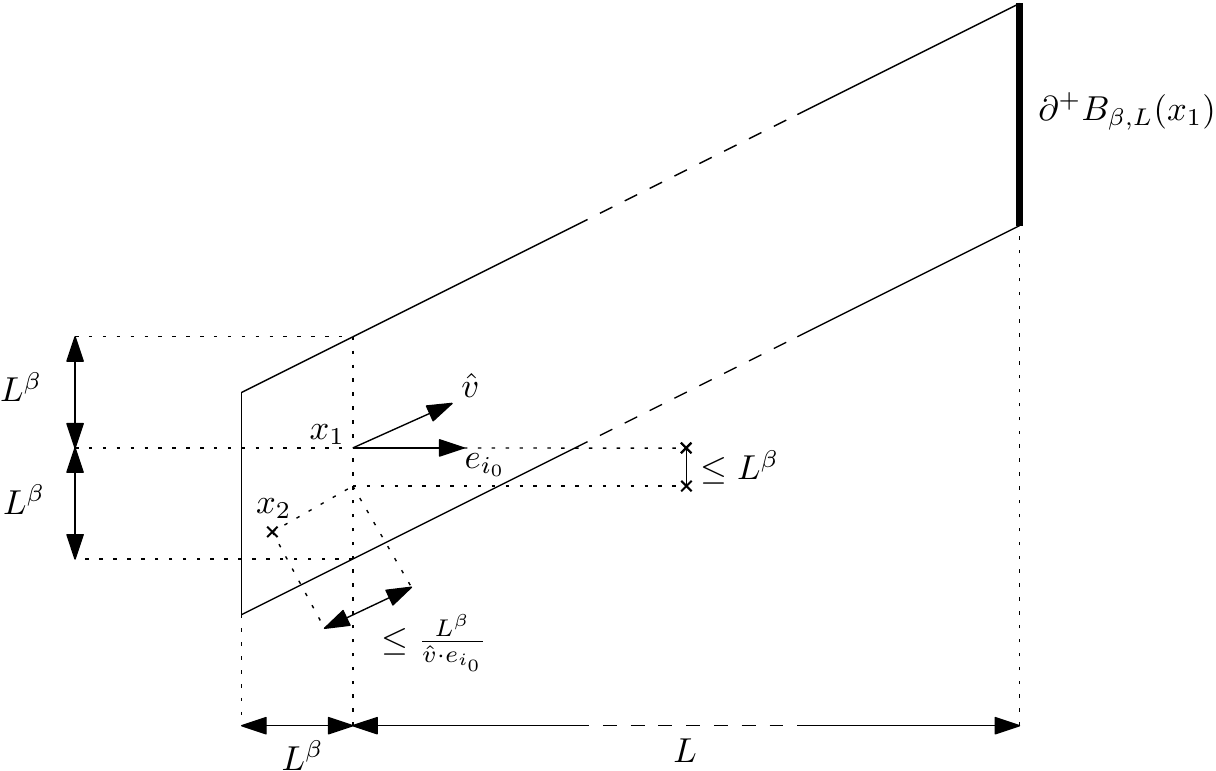}
      \caption{\label{dessin_tiltedbox} The tilted box $B_{\beta,L}(x_1)$.}

\end{figure}

Now, we have the following result from \cite{CR} that will be useful for us to show that it is extremely unlikely (super-exponential) that the environment typically  sends the walker against $\hat{v}$ for a long distance.
\begin{proposition}[Proposition 4.1 of Campos and Ram\'irez, \cite{CR}]\label{atypical}
Assume that $(E)_0$ holds and that $(P)_M^\ell$ is also satisfied for some $M\ge15d+5$. Let $\beta_0\in(1/2,1)$, $\beta\in(\frac{\beta_0+1}{2},1)$ and $\zeta\in(0,\beta_0)$. Then, for each $\gamma>0$, we have that
\[
\limsup_{L\rightarrow\infty} L^{-g(\beta_0,\beta,\zeta)}\log {\bf P}\left[P_0^\omega\left[X_{T_{B_{\beta,L}(0)}}\in\partial^+B_{\beta,L}(0)\right]\le e^{-\gamma L^\beta}\right]<0,
\]
where $g(\beta_0,\beta,\zeta)=\min\left\{\beta+\zeta, 3\beta-2+(d-1)(\beta-\beta_0)\right\}$.
\end{proposition}

In the next section, we will also need the following result that gives an equivalent criterion for $(T)_\gamma$, defined in \eqref{hyp_T}. For this purpose, define, for any $n\ge1$, the $n$-th regeneration radius by
\[
X^{*(n)}=\max_{\tau_{n-1}\le k \le\tau_n} |X_k-X_{\tau_{n-1}}|_1,
\]
where the $\tau_n$'s are defined in \eqref{regen_struct}.

The next result from Sznitman shows that if a walk verifies condition $(T)_\gamma^\ell$ then the trajectory of the walk goes fairly directly in the direction $\ell$, more precisely it is shown in~\cite{Szneff} that the space explored between two regeneration times has good tails. 
\begin{proposition}[Sznitman, \cite{Szneff}]\label{cor_sup}
Consider a RWRE in an elliptic i.i.d. environment. Let $\gamma\in(0,1)$ and $\ell\in S^{d-1}$. Assume that $(T)_\gamma^\ell$ holds. Then, there exists a constant $c$ such that, for every $L$ and $n\ge1$, we have that
\[
\PR\left[X^{*(n)}>L\right]\le Ce^{-cL^\gamma}.
\]
\end{proposition}


\subsection{Estimates on the tail of $\tau_1$: proof of Proposition \ref{tailtau}}\label{sect_tailtau}


Here, we prove Proposition \ref{tailtau}, which concludes the proof of Theorem \ref{main} and Theorem \ref{cor_TCL}.

\begin{proof}[Proof of Proposition \ref{tailtau}]
We want to give an estimate on the tail of $\tau_1$. For this purpose, we define, for $u>0$, the scale
\begin{align}\label{def_L}
L=L(u)=\left(c_1\eta \log u\right)^{\frac{1}{\beta}},
\end{align}
where $c_1\in(0,1)$, $\eta>0$ and $\beta\in(0,1)$ are constant which will be described later on. We also define the box 
\[
C_L=\left\{y\in\mathbb{Z}^d: -\frac{L}{2}<(y-x)\cdot e_{i_0}<\frac{L}{2}, ||Q(y-x)||_\infty< \frac{L}{2(\hat{v}\cdot e_{i_0})}\right\},
\]
using the definition \eqref{def_v} of $\hat{v}$ and recalling $\hat{v}\cdot e_{i_0}\ge 1/\sqrt{d}$.

Now, notice that
\begin{align}\label{ineg1}
\PR\left[\tau_1>u\right]\le \PR\left[\tau_1>u, T_{C_L}^{\text{ex}}\le\tau_1\right]+\PR\left[T_{C_L}^{\text{ex}}>u\right],
\end{align}
where $T_{C_L}^{\text{ex}}$ is the first time the walker is out of $C_L$.
Now, we want to give an upper bound for both of these quantities. For the first one, we will use condition $(T)_\gamma$ and we will use, for the second one, the atypical quenched exit estimates of Proposition \ref{atypical}.

\vspace{0.5cm}

{\it Upper bound for the first term of the right-hand side of \eqref{ineg1}}

\vspace{0.5cm}

First, we can give an estimate for the first quantity using Theorem \ref{PequivT} and Proposition \ref{cor_sup}, so that for any $\gamma\in(\beta,1)$, 
\begin{align}\label{estim1}
\PR\left[\tau_1>u, T_{C_L}^{\text{ex}}\le\tau_1\right]\le \PR\left[X^{*(1)}>\frac{L}{2\sqrt{d}}\right]\le Ce^{-cL^\gamma}.
\end{align}

\vspace{0.5cm}

{\it Upper bound for the second term of the right-hand side of \eqref{ineg1}}

\vspace{0.5cm}

Now, let us give an estimate for the second quantity of the right-hand side of \eqref{ineg1}. The general strategy is first to notice that, on the event $\left\{T_{C_L}^{\text{ex}}>u\right\}$, there exists some vertex $x\in C_L$ such that the probability starting from that point $x$ to come back to it before exiting $C_L$ is not too small. On the other hand, Corollary \ref{reachlog2} implies that there exists another point $y$, sufficiently far away from $x$, such that the probability to go from $x$ to $y$ without coming back to $x$ is great enough. These two facts together will imply that the quenched probability to  exit a tilted box (see \eqref{def_Btilted}) by the sides or the backside is large: this is an atypical quenched exit estimate, whose ${\bf P}$-probability is upper bounded by Proposition \ref{atypical}.\\

On the event $\left\{T_{C_L}^{\text{ex}}>u\right\}$, there is a.s.a random $x_1\in C_L$ such that
\begin{align}\label{def_Nx}
\mathcal{N}_{x_1}=\left|\left\{k:0\le k\le T_{C_L}^{\text{ex}},X_k=x_1\right\}\right|\ge \frac{u}{\left|C_L\right|},
\end{align}
which means that 
\[
\left\{T_{C_L}^{\text{ex}}>u\right\}\subset \left\{\exists x\in C_L: \mathcal{N}_x\geq \frac u{\abs{C_L}}\right\}.
\]

Note that, for any $x\in C_L$, if the walk starts from $x$, then $\mathcal{N}_x$ is a geometric random variable of parameter $P_x^\omega\left(T_{C_L}^{\text{ex}}<T_x^+\right)$, hence we get
\begin{align*}
\PR\left[\mathcal{N}_x\ge \frac{u}{|C_L|}, P_x^\omega\left[T_{C_L}^{\text{ex}}<T_x^+\right]\ge2\frac{\abs{C_L}}{u}L\right]\le e^{-L}.
\end{align*}

Using the two last equations, we see that
\begin{align}\label{estim2}
 \PR[T_{C_L}^{\text{ex}}>u] &\leq  \PR\left[T_{C_L}^{\text{ex}}>u, \inf_{x\in C_L} P_{x}^\omega\left[T_{C_L}^{\text{ex}}<T_{x}^+\right]<2\frac{\abs{C_L}}{u}L\right] \\ \nonumber
& \qquad + \PR\left[\exists x\in C_L: \mathcal{N}_x\geq \frac u{\abs{C_L}}, \inf_{x\in C_L} P_{x}^\omega\left[T_{C_L}^{\text{ex}}<T_{x}^+\right]\ge2\frac{\abs{C_L}}{u}L\right]\\\nonumber
& \leq    \PR\left[T_{C_L}^{\text{ex}}>u, \exists x_1 \in C_L: P_{x_1}^\omega\left[T_{C_L}^{\text{ex}}<T_{x_1}^+\right]<2\frac{\abs{C_L}}{u}L\right] \\ \nonumber
 & \qquad + \abs{C_L} \max_{x\in C_L} \PR\left[\mathcal{N}_x\ge \frac{u}{|C_L|}, P_x^\omega\left[T_{C_L}^{\text{ex}}<T_x^+\right]\ge2\frac{\abs{C_L}}{u}L\right] \\ \nonumber 
 & \le  \abs{C_L}e^{-L}+  \PR\left[T_{C_L}^{\text{ex}}>u, \exists x_1 \in C_L: P_{x_1}^\omega\left[T_{C_L}^{\text{ex}}<T_{x_1}^+\right]<2\frac{\abs{C_L}}{u}L\right] .
\end{align}

Let $\mathcal{A}$ be the event on which there exists $x_1\in C_L$ and a vertex $x_2\in \Z^d$ such that $\abs{x_2-x_1}_1= \lfloor \eta\log(u)\rfloor$ and
\begin{enumerate}
\item the following inequality holds
\[
P_{x_1}^\omega\left[T_{C_L}^{\text{ex}}<T_{x_1}^+\right]<2\frac{|C_L|}{u}L;
\]
\item   and the following inequality holds
\[
P^\omega_{x_1}\left[T_{x_2}<T^+_{x_1}\right]\ge u^{-\frac{\alpha+2\delta}{\alpha+\epsilon}}.
\]
\end{enumerate}

We can see that on $\mathcal{A}$ we have for $u$ large enough,
\begin{equation}\label{prop_A}
P^\omega_{x_1}\left[T_{x_2}<T^+_{x_1}\wedge T_{C_L}^{\text{ex}}\right]\ge \frac{1}{2}u^{-\frac{\alpha+2\delta}{\alpha+\epsilon}},
\end{equation}
which, in particular, implies that $x_2\in C_L$.

Besides, recall that $(K)_\alpha$ holds and that some $\epsilon>0$ is associated with that condition. Then, fixing $\delta\in(0,\epsilon/4)$, by Corollary \ref{reachlog2} there exists $\eta>0$ small enough such that, as soon as $u$ is large enough,
\begin{align}\label{estim3}
\PR\left[\exists x\in C_L: \max_{y:|y-x|_1=\lfloor \eta\log(u)\rfloor} P^\omega_x\left[T_y<T^+_x\right]< u^{-\frac{\alpha+2\delta}{\alpha+\epsilon}}\right]\le \frac{|C_L|}{u^{\alpha+\delta}}.
\end{align}

Using \eqref{estim2} and \eqref{estim3}, we get the upper bound:
\begin{align}\label{estim11}
 \PR[T_{C_L}^{\text{ex}}>u]  \le \abs{C_L}\left(u^{-(\alpha+\delta)}+e^{-L}\right) + \PR\left[\mathcal{A},T_{C_L}^{\text{ex}}>u\right].
\end{align}

Note also that, as soon as $u$ is large enough and using \eqref{def_L},
\begin{align}\label{dist1}
|x_2-x_1|_1=\lfloor \eta\log(u)\rfloor \ge \frac{\eta}{2} \log(u)=\frac{1}{2c_1}L^{\beta},
\end{align}
recalling that $c_1\in(0,1)$ is a constant that we will fix later on.\\

Now, let us consider the tilted box $B_{\beta,L}(x_1)$ defined in \eqref{def_Btilted} with $c_1<(4d^2)^{-1}$ and distinguish two cases on the event $\mathcal{A}$. 

\vspace{0.5cm}

{\it First case: on the event $\mathcal{A}\cap\left\{x_2\in B_{\beta,L}(x_1)\right\}$}

\vspace{0.5cm}

First, if $x_2\in B_{\beta,L}(x_1)$, we will prove that $x_1\notin B_{\beta,L}(x_2)$. 

Let us assume by contradiction that $x_1\in B_{\beta,L}(x_2)$, we can see that by Remark~\ref{elem_geom} (which can always be applied since $(x_1-x_2)\cdot \hat v \ge0$ or $(x_2-x_1)\cdot \hat v \ge0$, and $x_1$ and $x_2$ play symmetric roles for this computation)
\[
\abs{\abs{x_1-x_2}}_1\leq d(1+\sqrt d)L^{\beta}.
\]

The previous equation contradicts~\eqref{dist1} since we just chose $c_1<(4d^2)^{-1}$. Hence $x_1\notin B_{\beta,L}(x_2)$.

Moreover, one has that
\begin{align*}
P_{x_1}^\omega\left[T_{C_L}^{\text{ex}}<T^+_{x_1}\right]&\ge P_{x_1}^\omega\left[T_{x_2}<T_{C_L}^{\text{ex}}\wedge T^+_{x_1}\right] \times P_{x_2}^\omega\left[T_{C_L}^{\text{ex}}<T_{x_1}\right]\\
&\ge \frac{1}{2u^{\frac{\alpha+2\delta}{\alpha+\epsilon}}}\times P_{x_2}^\omega\left[T_{C_L}^{\text{ex}}<T_{x_1}\right],
\end{align*}
where we used the fact that we are on $\mathcal{A}$ (see~\eqref{prop_A}). Furthermore, the definition of $\mathcal{A}$, then implies
\[
P_{x_2}^\omega\left[T_{C_L}^{\text{ex}}<T_{x_1}\right]\le4 \frac{|C_L|\times L}{u^{\frac{\epsilon}{2(\alpha+\epsilon)}}}.
\]

\begin{figure}[h]
   \includegraphics{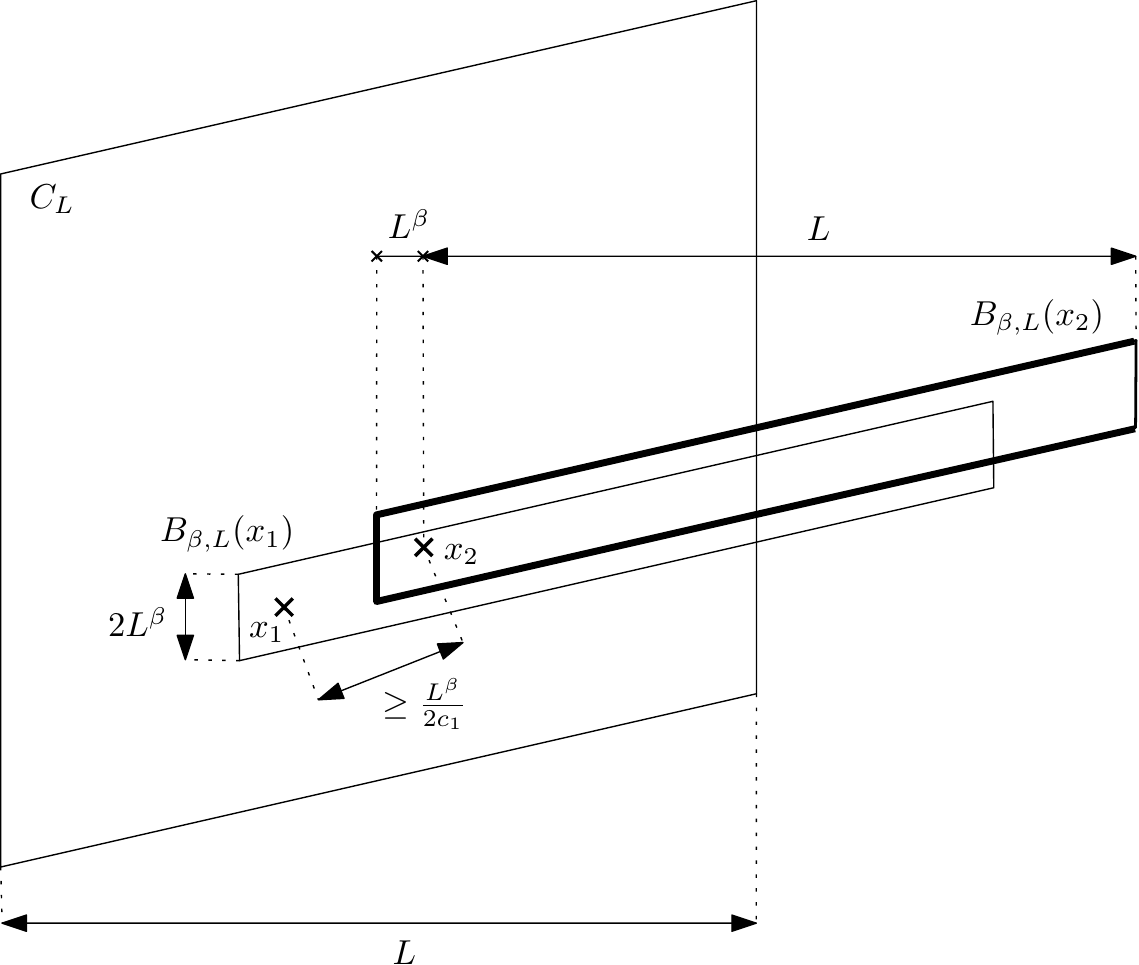}
      \caption{\label{dessin_x11} From $x_2$, the walker has to exit $B_{\beta,L}(x_2)$ before visiting $x_1$.}

\end{figure}

This last inequality implies that the probability, starting from $x_2$, to exit $C_L$ before visiting $x_1$ is very small. This fact implies that the probability to exit $B_{\beta,L}(x_2)$ through its front boundary is very small as well (see Figure~\ref{dessin_x11}).
\begin{align}
P_{x_2}^\omega\left[X_{T_{B_{\beta,L}(x_2)}}\in\partial^+B_{\beta,L}(x_2)\right]&\le P_{x_2}^\omega\left[T_{C_L}^{\text{ex}}<T_{x_1}\right]\le4 \frac{|C_L|\times L}{u^{\frac{\epsilon}{2(\alpha+\epsilon)}}}\label{estim41}.
\end{align}

\vspace{0.5cm}

{\it Second case: on the event $\mathcal{A}\cap\left\{x_2\notin B_{\beta,L}(x_1)\right\}$}

\vspace{0.5cm}

If $x_2$ does not belong to $B_{\beta,L}(x_1)$, then it is obvious that the walker cannot visit $x_2$ without exiting the tilted box (see Figure~\ref{dessin_x12}). This means that 
\[
 P_{x_1}^\omega\left[T_{x_2}<T_{C_L}^{\text{ex}}\wedge T_{x_1}^+\right] \leq P_{x_1}^\omega\left[T_{B_{\beta,L}(x_1)}<T_{x_1}^+\right].
\]

The walker cannot visit too many times $x_1$ before exiting $B_{\beta,L}(x_1)$, indeed the walker goes relatively easily from $x_1$ to $x_2$ and  $x_2$ can only be reached by exiting  $B_{\beta,L}(x_1)$.

\begin{figure}[h]
   \includegraphics{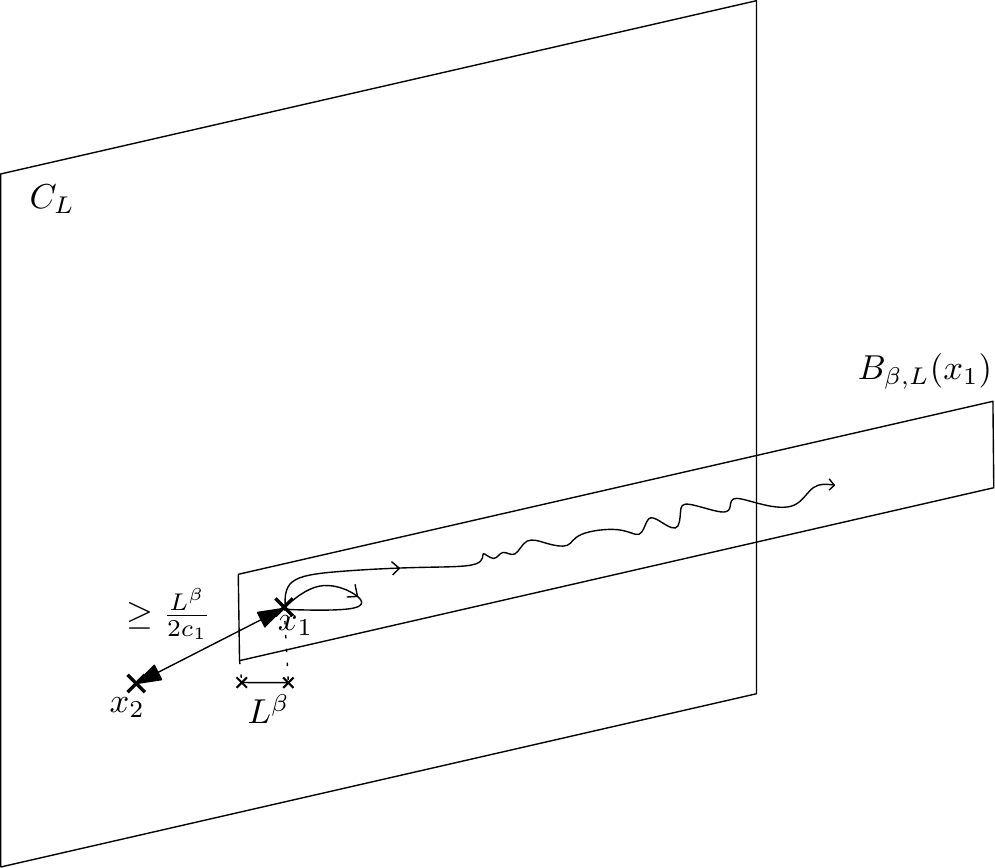}
      \caption{\label{dessin_x12} From $x_1$, the walker has to exit $B_{\beta,L}(x_1)$ before visiting $x_2$.}

\end{figure}

More precisely,  recalling the definition \eqref{def_Nx} of $\mathcal{N}_{x_1}$ and using the previous equation in the third line, we see that on $\mathcal{A}$
\begin{align}
P_{x_1}^\omega\left[X_{T_{B_{\beta,L}(x_1)}}\in\partial^+B_{\beta,L}(x_1)\right]&\le \sum_{n=0}^\infty P_{x_1}^\omega\left[X_{T_{B_{\beta,L}(x_1)}}\in\partial^+B_{\beta,L}(x_1),\mathcal{N}_{x_1}=n+1\right]\nonumber\\
&\le P_{x_1}^\omega\left[ T_{C_L}^{\text{ex}}<T^+_{x_1}\right]\times \sum_{n=0}^\infty \left[P_{x_1}^\omega\left[T_{x_1}^+<T_{B_{\beta,L}(x_1)}\right]\right]^n\nonumber\\
\le P_{x_1}^\omega&\left[ T_{C_L}^{\text{ex}}<T^+_{x_1}\right]\times \sum_{n=0}^\infty \left(1-P_{x_1}^\omega\left[T_{x_2}<T_{C_L}^{\text{ex}}\wedge T_{x_1}^+\right]\right)^n\nonumber\\
&\le \frac{P_{x_1}^\omega\left[ T_{C_L}^{\text{ex}}<T^+_{x_1}\right]}{P_{x_1}^\omega\left[T_{x_2}<T_{C_L}^{\text{ex}}\wedge T_{x_1}^+\right]}\le 4\frac{|C_L|\times L}{u^{\frac{\epsilon}{2(\alpha+\epsilon)}}}.\label{estim42}
\end{align}

\vspace{0.5cm}

{\it Atypical quenched exit estimates on $\mathcal{A}$}

\vspace{0.5cm}

By \eqref{estim41} and \eqref{estim42}, we have that, on $\mathcal{A}$, there a.s.~exists $x\in C_L$ such that, for some positive constant $c_3$, we have that, as soon as $L$ is large enough,
\[
P_{x}^\omega\left[X_{T_{B_{\beta,L}(x)}}\in\partial^+B_{\beta,L}(x)\right]\le e^{-c_3 L^{\beta}}.
\]

Let us stress that the constants $c_1$ and $\eta$ have been fixed, so we will not emphasize them in the following computations. By using Proposition \ref{atypical}, we obtain a function $g(\beta_0,\beta,\zeta)$ such that 
\begin{align*}
\PR\Bigl[P_{x}^\omega\left[X_{T_{B_{\beta,L}(x)}}\in\partial^+B_{\beta,L}(x)\right]\le e^{-c_3 L^{\beta}}\Bigr] & \leq C\exp\left(-c L^{g(\beta_0,\beta,\zeta)}\right) \\
& \leq C\exp\left(-c (\log u)^{\frac{g(\beta_0,\beta,\zeta)}{\beta}}\right),
\end{align*}
where we recall that $L$ was defined at \eqref{def_L}.

An elementary computation shows that $g(\beta_0,\beta,\zeta)>\beta$ for $\beta$ close to 1, $\beta_0$ close to $1/2$ and $\zeta>0$. Thus, for such a choice of constants, there exists $\epsilon'>0$ such that for $u$ is large enough,
\begin{align}\label{estim4}
\PR[\mathcal{A}]\le e^{-c_4 (\log u)^{1+\epsilon'}}.
\end{align}

\vspace{0.5cm}

{\it Conclusion}

\vspace{0.5cm}

The inequality \eqref{ineg1} and the estimates \eqref{estim1}, \eqref{estim11} and \eqref{estim4}, we conclude that, as soon as $u$ is large enough,
\[
\PR[\tau_1>u]\le \frac{1}{u^{\alpha+\frac{\delta}{2}}},
\]
for some $\delta>0$.
\end{proof}

\begin{remark}
Notice that in this last proof, the only limiting factor that prevents us to obtain moments of any order on $\tau_1$ is \eqref{estim3} which describes the probability to reach a certain point at distance of order $\log n$.\\
\end{remark}


\section{Zero-speed regime}
\label{sect_zero_speed}

In this section, let us prove Theorem~\ref{lb_tail}. To accomplish this, we need to identify where trapping comes from.

For this, we say that a vertex $x\in\Z^d$ is $\kappa$-elliptic if for all $e\in U$, we have $p^{\omega}(x,e) \in (\kappa,1-\kappa)$. By \eqref{hyp_E}~it it clear that there exists $\kappa_0>0$ such that ${\bf P}[x\text{ is $\kappa_0$-elliptic}]>1/2$ for any $x\in \Z^d$. To be concise, we will say that a vertex is regular if it is $\kappa_0$-elliptic.

Let us introduce the sets
 \[
 \mathcal{A}= \{z\notin \mathfrak{H}_{de_1}, \text{such that for some $y\in\mathfrak{H}_{de_1}$, }\abs{\abs{z-y}}_{\infty}=1\}.
 \]
 and  
 \[
 \mathcal{B}=\{0,\ e_1,\ldots,\ (d-1)e_1\}.
 \]
 
It is plain to see that
\begin{enumerate}
\item  $\mathcal{A}\cup\mathcal{B}$ is connected, 
 \item $\mathcal{A}$ contains $\partial \mathfrak{H}_{de_1}$,
 \item $\mathcal{A}\subset \mathcal{H}^+(0)$ (defined at~\eqref{def_HH}). This can be seen easily from~(\ref{e1ell}).
\end{enumerate}

Let us introduce the event
\begin{equation}\label{def_R}
\mathcal{R}=\{\text{any $x\in \mathcal{A}\cup \mathcal{B}$ is regular}\},
\end{equation}
it is clear that ${\bf P}[\mathcal{R}]>0$.

The general idea is to investigate the probability of events such that some unit hypercube is surrounded by regular points, but transition probabilities inside the hypercube are not conditioned. Thus, on such an event, the walker moves easily around the hypercube but could get trapped in it, as the exit time of the hypercube is not conditioned and is independent of the environment outside it (see Figure \ref{dessin_0}).
\begin{figure}[h]
   \includegraphics{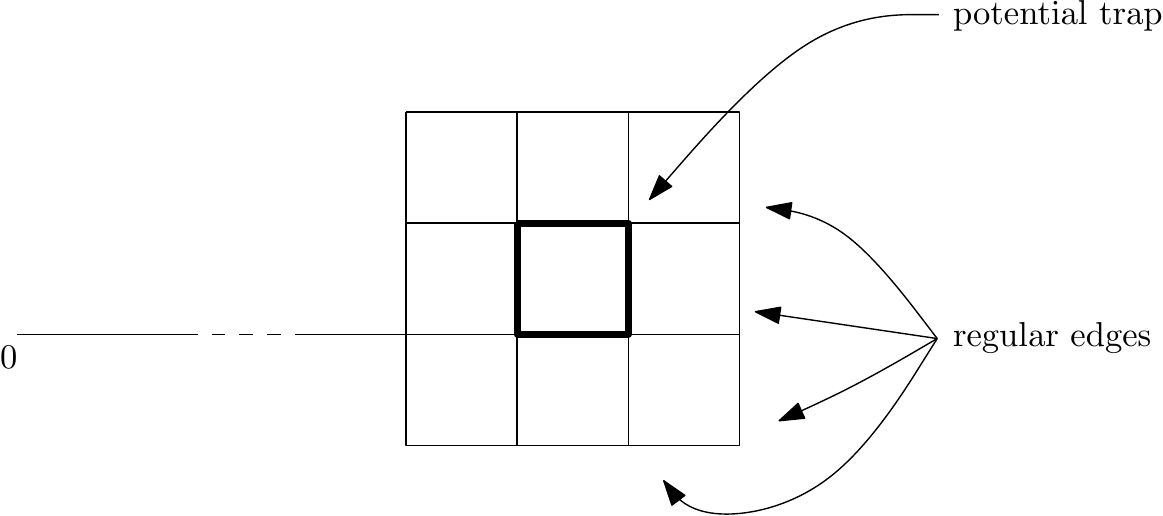}
      \caption{\label{dessin_0} The environment $\mathcal{R}$ on which we construct trapping.}

\end{figure}

The following lemma shows that tail estimates on the exit time of hypercubes can be used to find lower bounds on regeneration times.
\begin{lemma}\label{tail_fromtrap}
Let us consider a RWRE in an elliptic i.i.d.~environment. We have, for some constant $c>0$,
\[
\PR[\tau_1\geq n\mid 0-\text{regen}] \geq c \max_{x\in \mathfrak{H}} \PR_x[T^{\text{ex}}_{\mathfrak{H}}\geq n].
\]
\end{lemma}
\begin{proof}
We fix $x_0\in \mathfrak{H}$ which realizes the maximum $\max_{x\in \mathfrak{H}} \PR_x[T^{\text{ex}}_{\mathfrak{H}}\geq n]$.

Let us now describe an event which slows the walk down and which can happen on $\{0-\text{regen}\}$. On $\mathcal{R}$, consider the following chain of events
\begin{enumerate} 
\item $X_1=e_1, X_2=2e_1,\ldots, X_d=(d-1)e_1$,
\item from there $X_n$ takes the shortest path inside $\mathcal{A}\cup\mathcal{B}$ to $de_1+x_0$, this can be done in at most $C(d)$ steps (where $C(d)$ depends only on $d$).
\item Then, we stay on $\mathfrak{H}_{de_1}$ for a time $T^{\text{ex}}_{\mathfrak{H}_{de_1}}\circ \theta_{T_{de_1+x_0}} \geq n$, where $\theta_{\cdot}$ is a shift operator,
\item after exiting $\mathfrak{H}_{de_1}$, the walk has to be in $\mathcal{A}$. From there, the walk takes the shortest path to $e_1$ inside $\mathcal{A}\cup\mathcal{B}$ and then the shortest path from $e_1$ to $de_1+2\sum_{i=1}^d e_i$ (which has never been visited) inside $\mathcal{A}\cup\mathcal{B}$. All this can be accomplished in less than $C(d)$ steps (where $C(d)$ depends only on $d$). This step ensures that $\tau_1$ occurs after $T^{\text{ex}}_{\mathfrak{H}_{de_1}}\circ \theta_{T_{de_1+x_0}}$.
\item Finally, the walk makes one step to $(d+1)e_1+2\sum_{i=1}^d e_i$,  and from there never backtracks, creating a new regeneration time. 
\end{enumerate}

Let us denote $F_n$ the chain of events described above (in (1), (2), (3), (4), (5)). We can see that on $F_n$, we have
\begin{enumerate}
\item $D=\infty$,
\item $\tau_1\ge T^{\text{ex}}_{\mathfrak{H}_{de_1}}\circ \theta_{T_{de_1+x_0}} +T_{de_1+x_0}\geq n$.
\end{enumerate}

This implies that 
\begin{align}\label{tau1_trap}
\PR[\tau_1\geq n \mid 0-\text{regen}] & \geq c\PR[\tau_1\geq n,\ 0-\text{regen}] \\ \nonumber
                                                                  & \geq c{\bf E}[\1{\mathcal{R}} P^{\omega}_0[F_n]].
\end{align}

Now, we want to give a lower bound of $P^{\omega}[F_n]$ on the event $\mathcal{R}$. This can be done by applying several times the strong Markov property at the times $d-1$, $T_{de_1+x_0}$, $T^{\text{ex}}_{\mathfrak{H}_{de_1}}\circ \theta_{T_{de_1+x_0}}+T_{de_1+x_0}$, $T_{de_1+2\sum_{i=1}^d e_i}$. This leads us to lower-bound the five terms described above.
\begin{enumerate}
\item On $\mathcal{R}$, we have that 
\[
P_0^{\omega}[X_1=e_1, X_2=2e_1,\ldots, X_{d-1}=(d-1)e_1] \geq \kappa_0^{d-1}.
\]
\item Let us denote $C_1$ the event that the walk takes the shortest path from $(d-1)e_1$ to $x_0+de_1$ inside $\mathcal{A}\cup\mathcal{B}$. On $\mathcal{R}$, we have that 
\[
P_{(d-1)e_1}^{\omega}[C_1] \geq \kappa_0^{C(d)}.
\]
\item After applying the Markov property, the third term becomes $P_{de_1+x_0}^{\omega}[T^{\text{ex}}_{\mathfrak{H}_{de_1}}\geq n]$.
\item Let us denote $C_2$ the event that the walk takes the shortest path to $e_1$ inside $\mathcal{A}\cup\mathcal{B}$ and then the shortest path from $e_1$ to $de_1+2\sum_{i=1}^d e_i$ inside $\mathcal{A}\cup\mathcal{B}$. It is easy to see, on $\mathcal{R}$, that
\[
\min_{y\in \partial \mathfrak{H}} P_y^{\omega}[C_2]\geq \kappa_0^{C(d)}.
\]
\item Finally, we see that, on $\mathcal{R}$,
\[
P_{ de_1+2\sum_{i=1}^d e_i}^{\omega}[X_1=(d+1)e_1+2\sum_{i=1}^d e_i,\ D\circ \theta_{1}=\infty] \geq \kappa_0 P_{(d+1)e_1+2\sum_{i=1}^d e_i}^{\omega}[D=\infty].
\]
\end{enumerate}

As mentioned, those estimates combined with the strong Markov property imply that on $\mathcal{R}$, we have
\[
P_0^{\omega}[F_n] \geq c(\kappa_0,d) P_{de_1+x_0}^{\omega}[T^{\text{ex}}_{\mathfrak{H}_{de_1}}\geq n] P_{(d+1)e_1+2\sum_{i=1}^d e_i}^{\omega}[D=\infty].
\]

This estimate combined with~(\ref{tau1_trap}) implies that
\[
\PR[\tau_1\geq n \mid 0-\text{regen}] \geq c {\bf E}\Bigl[\1{\mathcal{R}}P_{de_1+x_0}^{\omega}[T^{\text{ex}}_{\mathfrak{H}_{de_1}}\geq n] P_{(d+1)e_1+2\sum_{i=1}^d e_i}^{\omega}[D=\infty]\Bigr].
\]

Note that by independence of the transition probabilities, the random variables $P_{de_1+x_0}^{\omega}[T^{\text{ex}}_{\mathfrak{H}_{de_1}}\geq n]$, $\1{\mathcal{R}}$ and $P_{(d+1)e_1+2\sum_{i=1}^d e_i}^{\omega}[D=\infty]$ are all ${\bf P}$-independent. This means that
\begin{align*}
\PR[\tau_1\geq n \mid 0-\text{regen}] & \geq c {\bf P}[\mathcal{R}] {\bf E}\Bigl[P_{de_1+x_0}^{\omega}[T^{\text{ex}}_{\mathfrak{H}_{de_1}}\geq n]\Bigr] {\bf E}\Bigl[P_{(d+1)e_1+2\sum_{i=1}^d e_i}^{\omega}[D=\infty]\Bigr] \\
 &\geq c \PR_{x_0}[T^{\text{ex}}_{\mathfrak{H}}\geq n]  \PR[D=\infty],
 \end{align*}
 where we used translation invariance and the fact that ${\bf P}[\mathcal{R}]>0$ in the last line. The result follows from the definition of $x_0$ and~(\ref{pos_D}).

\end{proof}

Let us now prove Theorem~\ref{lb_tail}.
\begin{proof}[Proof of Theorem \ref{lb_tail}]
It is easy to see from Lemma~\ref{tail_fromtrap} that a RWRE in an elliptic i.i.d.~environment verifies
  \[
\ES[\tau_1^{\alpha}\mid 0-\text{regen}] \geq c \max_{x\in \mathfrak{H}} \ES_x\Bigl[\Bigl(T^{\text{ex}}_{\mathfrak{H}}\Bigr)^\alpha\Bigr].
\]

This means that for a walk verifying $(E)$ and $(H)_{\alpha}$ for some $\alpha>0$ we have
\begin{equation}\label{proof_rem}
\ES[\tau_1^{\alpha}\mid 0-\text{regen}]=\infty.
\end{equation}

Theorem~\ref{lb_tail} follows from the previous equation and Theorem~\ref{speed_regen}.
\end{proof}


\noindent{\bf Acknowledgments.}

We would like to thank Alejandro Ram\'irez for useful discussions.

The authors are also grateful to the Universit\'e de Toulouse, which they were both affiliated to at the time when this work was done.

\end{document}